\documentclass[11pt]{article}   
\usepackage{amsthm,amsfonts,amsmath,amscd,amssymb,latexsym}    
\usepackage{mathrsfs}
\usepackage{epsfig,graphicx,color}     
\usepackage{color}    
\usepackage[all]{xy}  
\usepackage[titles]{tocloft}  
\newif\ifprintable
\printablefalse
\ifprintable

\newcommand{\url}[1]{{#1}}

\else
\usepackage[colorlinks=true]{hyperref}
\fi
\makeindex
\newif\ifmarking
\markingtrue
\title{The three-dimensional Fueter equation\\
and divergence-free frames}   
 
\author{Dietmar~Salamon\thanks{Partially supported 
by the Swiss National Science Foundation Grant 200021-127136} \\
ETH Z\"urich} 
 
\date{25 February 2013}   

\newtheorem{PARA}{}[section] 
\newtheorem{theorem}[PARA]{Theorem} 
 
\newtheorem{lemma}[PARA]{Lemma} 
 
\newtheorem{definition}[PARA]{Definition}    
\newtheorem{remark}[PARA]{Remark} 
\newtheorem{example}[PARA]{Example} 
\newtheorem{conjecture}[PARA]{Conjecture} 
\newcommand{\sA}{\mathscr{A}}    
\newcommand{\sB}{\mathscr{B}}    
    
\newcommand{\sD}{\mathscr{D}}    
\newcommand{\sE}{\mathscr{E}}    
\newcommand{\sF}{\mathscr{F}}    
\newcommand{\sG}{\mathscr{G}}    
\newcommand{\sH}{\mathscr{H}}

\newcommand{\sL}{\mathscr{L}}

\newcommand{\sP}{\mathscr{P}}    
    
\newcommand{\sR}{\mathscr{R}}    
\newcommand{\sS}{\mathscr{S}}

\newcommand{\sV}{\mathscr{V}}

\newcommand{\tsF}{\widetilde{\mathscr{F}}}   
\newcommand{\cA}{\mathcal{A}}    
\newcommand{\cB}{\mathcal{B}}

\newcommand{\cF}{\mathcal{F}}    
\newcommand{\cE}{\mathcal{E}}    
\newcommand{\cG}{\mathcal{G}}

\newcommand{\cL}{\mathcal{L}}    
\newcommand{\cM}{\mathcal{M}}    
    
\newcommand{\cP}{\mathcal{P}}    
    
\newcommand{\cR}{\mathcal{R}}

\newcommand{\cX}{\mathcal{X}}

\newcommand{\A}{{\mathbb{A}}}

\renewcommand{\H}{{\mathbb{H}}}

\newcommand{\Q}{{\mathbb{Q}}}    
\newcommand{\R}{{\mathbb{R}}}    
\newcommand{\T}{{\mathbb{T}}}    
    
\newcommand{\Z}{{\mathbb{Z}}}    
\newcommand{\CS}{{\mathrm{CS}}}    
\newcommand{\CSD}{{\mathrm{CSD}}}    
\newcommand{\cs}{{\mathrm{cs}}}

\newcommand{\CFHK}{\mathrm{CF}^{\mathrm{hk}}}  
\newcommand{\HFHK}{\mathrm{HF}^{\mathrm{hk}}}    
    
\newcommand{\HFDT}{\mathrm{HF}^{\mathrm{DT}}}      
    
\newcommand{\Sp}{{\mathrm{Sp}}}

\newcommand{\Map}{{\mathrm{Map}}}

\newcommand{\tsD}{{\widetilde\sD}}

\newcommand{\tsH}{{\widetilde\sH}}

\newcommand{\tL}{{\widetilde L}}    
\newcommand{\tM}{{\widetilde M}}

\newcommand{\tsR}{{\widetilde\sR}}    
    
\newcommand{\tsS}{{\widetilde\sS}}    
    
\newcommand{\tsV}{{\widetilde\sV}}    
\newcommand{\tX}{{\widetilde X}}

\newcommand{\tg}{{\widetilde g}}

\newcommand{\tpi}{{\widetilde\pi}}

\newcommand{\fhat}{{\widehat f}}    
\newcommand{\vhat}{{\widehat v}}    
\renewcommand{\i}{{\mathbf{i}}}    
\renewcommand{\j}{{\mathbf{j}}}    
\renewcommand{\k}{{\mathbf{k}}}    
\newcommand{\one}    
{{{\mathchoice \mathrm{ 1\mskip-4mu l} \mathrm{ 1\mskip-4mu l}    
\mathrm{ 1\mskip-4.5mu l} \mathrm{ 1\mskip-5mu l}}}}    
\def\slashi#1{\rlap{\sl/}#1}
\def\slashii#1{\setbox0=\hbox{$#1$}             % set a box for #1 
\dimen0=\wd0                                 % and get its size
\setbox1=\hbox{\sl/} \dimen1=\wd1            % get size of /
\ifdim\dimen0>\dimen1                        % #1 is bigger
\rlap{\hbox to \dimen0{\hfil\sl/\hfil}}   % so center / in box
#1                                        % and print #1
\else                                        % / is bigger
\rlap{\hbox to \dimen1{\hfil$#1$\hfil}}   % so center #1
\hbox{\sl/}                               % and print /
\fi}                                         %
\def\slashiii#1{\setbox0=\hbox{$#1$}#1\hskip-\wd0\hbox to\wd0{\hss\sl/\/\hss}}
\newcommand{\dd}{{\slashi{\partial}}}

\newcommand{\dD}{{\slashiii{D}}}

\newcommand{\dslash}{/\mskip-6mu/}    
\newcommand{\ch}{\mathrm{ ch}}    
    
\newcommand{\dvol}{\mathrm{ dvol}} 
  
\newcommand{\coker}{\mathrm{ coker }}

\newcommand{\trace}{\mathrm{ trace }}

\newcommand{\odd}{{\mathrm{odd}}}

\newcommand{\id}{\mathrm{ id}} 
    
\newcommand{\im}{\mathrm{ im }}

\newcommand{\INDEX}{\mathrm{ index}}

\newcommand{\Hom}{\mathrm{Hom}}

\newcommand{\G}{{\mathrm{G}}}     
\newcommand{\GL}{{\mathrm{GL}}}

\newcommand{\SO}{{\mathrm{SO}}}             
\newcommand{\Spin}{\mathrm{ Spin}}       
\newcommand{\g}{{\mathfrak{g}}}

\newcommand{\Lie}{{\mathrm{Lie}}}

\newcommand{\Diff}{\mathrm{ Diff}}  
\newcommand{\End}{\mathrm{ End}}     
 
\newcommand{\Vect}{\mathrm{ Vect}}     
\newcommand{\eps}{{\varepsilon}}    
\newcommand{\om}{{\omega}}    
\newcommand{\Om}{{\Omega}}    
\newcommand{\Cinf}{C^{\infty}}    
\newcommand{\reg}{\mathrm{ reg}}    
\newcommand{\sing}{\mathrm{ sing}}    

\newcommand{\inner}[2]{\bigl\langle #1, #2\bigr\rangle}    
\newcommand{\Inner}[2]{\left\langle #1, #2\right\rangle}    
\newcommand{\winner}[2]{\left\langle #1{\wedge}#2\right\rangle}    
\def\NABLA#1{{\mathop{\nabla\kern-.5ex\lower1ex\hbox{$#1$}}}}    
\def\Nabla#1{\nabla\kern-.5ex{}_{#1}}    
\def\Tabla#1{\Tilde\nabla\kern-.5ex{}_{#1}}    
\def\abs#1{\mathopen|#1\mathclose|}    
\def\Abs#1{\left|#1\right|}    
    
\def\Norm#1{\left\|#1\right\|}    
\renewcommand{\Tilde}{\widetilde}

\newcommand{\p}{{\partial}}

\begin{document}    
    
\maketitle    

%%%%%%%%%%%%%%%%%%%%%%%%%%%%%%%%%%%%%%%%%%%%%%    
%%%%%%%%%%%%%%%%%%%%%%%%%%%%%%%%%%%%%%%%%%%%%%    
%%%%%%%%%%%%%%%% Abstract %%%%%%%%%%%%%%%%%%%%    
%%%%%%%%%%%%%%%%%%%%%%%%%%%%%%%%%%%%%%%%%%%%%%    
%%%%%%%%%%%%%%%%%%%%%%%%%%%%%%%%%%%%%%%%%%%%%%    
    
%\vspace{-.8cm}  
\begin{abstract}  
This paper extends hyperk\"ahler Floer theory 
with flat target mani\-folds from the case where the 
source is a $3$-sphere or $3$-torus, equipped with 
a standard frame, to the case where the source is a 
general closed orientable $3$-manifold, equipped with 
a {\it regular} divergence-free frame. 
Regular divergence-free frames are characterized 
by the nonexistence of nonconstant solutions to the 
unperturbed linear Fueter equation.  They form a dense 
open subset of the space of all divergence-free frames.
A gauged version of the Fueter equation is introduced,
which unifies various geometric equations in gauge theory. 
\end{abstract}

%%%%%%%%%%%%%%%%%%%%%%%%%%%%%%%%%%%%%%%%%%%%%%    
%%%%%%%%%%%%%%%%%%%%%%%%%%%%%%%%%%%%%%%%%%%%%%    
%%%%%%%%%%%%%%%% Section 1 %%%%%%%%%%%%%%%%%%%    
%%%%%%%%%%%%%%%%%%%%%%%%%%%%%%%%%%%%%%%%%%%%%%    
%%%%%%%%%%%%%%%%%%%%%%%%%%%%%%%%%%%%%%%%%%%%%%    
   
\section{Introduction} \label{sec:INTRO}

The equation in the title was introduced by Rudolph Fueter
in his study of analytic functions of one quaternionic variable
in the 1930's~\cite{Fu1,Fu2}.  The three-dimensional reduction
of the Fueter equation carries over to functions $f:M\to X$ from 
any three-manifold $M$ (equipped with a volume form and a global 
divergence-free frame $v_1,v_2,v_3$) to any hyperk\"ahler 
manifold $X$ (with complex structures $J_1,J_2,J_3$).  
It has the form
\begin{equation}\label{eq:FUETER}
J_1\p_{v_1}f + J_2\p_{v_2}f+J_3\p_{v_3}f=\nabla H(f).
\end{equation}
The function $H:M\times X\to\R$ determines a zeroth order
perturbation.  There is a natural analogy between the solutions
of~\eqref{eq:FUETER} and periodic orbits of Hamiltonian
systems in a symplectic manifold. The solutions 
of~\eqref{eq:FUETER} are critical points of an action functional,
there is a Floer type theory for this functional, and an analogue
of the Arnold conjecture holds in favorable~cases.
For flat target manifolds and standard frames 
on $S^3$ and $\T^3$ hyperk\"ahler Floer theory 
was developed in~\cite{HNS1,HNS2} and the Arnold 
conjecture was derived as a corollary.  With different methods 
the hyperk\"ahler Arnold conjecture was extended by 
Ginzburg--Hein~\cite{GH1,GH2} to a more general setting.

The starting point for the present paper was the question 
under which condition the hyperk\"ahler Floer theory in~\cite{HNS1,HNS2}
extends to general three-manifolds and divergence-free frames.
The key point is an estimate which asserts that the $L^2$-norm
of the sum 
$
\dd_vf:=J_1\p_{v_1}f + J_2\p_{v_2}f+J_3\p_{v_3}f
$
controls the $L^2$ norm of $df$ (Lemma~\ref{le:estimate}).  
Such an estimate holds if and only if the linear Fueter equation
with target space $X=\H$ (the quaternions) and zero Hamiltonian 
has no nonconstant solutions.  If nonconstant solutions 
do exist, compactness fails for the solutions of equation~\eqref{eq:FUETER}.
Call a divergence-free frame $v$ {\it regular} if
every solution $f:M\to\H$ of the linear Fueter equation 
$\dd_vf=0$ is constant and call it {\it singular} otherwise.   
Then the hyperk\"ahler Floer theory with flat target manifolds 
carries over to all regular divergence-free frames and the 
hyperk\"ahler Arnold conjecture holds in this case.

Section~\ref{sec:FUETER} discusses the space $\sV$
of divergence-free frames on a closed three-manifold
and the linear Fueter operator.  It is shown that the set
$\sV^\reg$ of regular divergence-free frames is open 
and dense in $\sV$ (Lemma~\ref{le:regular}), 
that the set $\sV_1$ of divergence-free frames 
$v\in\sV$ with $\dim(\ker\,\dd_v)=8$ is a codimension
one submanifold of $\sV$ (Lemma~\ref{le:V1}), and the 
topology of $\sV$ is examined (Lemma~\ref{le:topology}). 
Section~\ref{sec:EX} discusses several examples. 
It is shown that singular divergence-free
frames exist on $M=S^3$ and that the standard frame on 
$M=S^2\times S^1$ is regular. Section~\ref{sec:HF} explains 
how the hyperk\"ahler Floer theory of~\cite{HNS1,HNS2} 
extends to regular divergence-free frames.

Section~\ref{sec:GAUGE} introduces the {\it gauged Fueter equation},
associated to a $\G$-action on $X$ generated by a hyperk\"ahler 
moment map $\mu=(\mu_1,\mu_2,\mu_3):X\to\g^3$.
The four-dimensional version of this equation has the form
\begin{equation}\label{eq:Gfueter}
\p_su+L_u\Phi-\dd_{A,v}u = 0,\qquad
\p_sA-d_A\Phi + *F_A 
= \sum_i(\mu_i\circ u)\pi^*\alpha_i
\end{equation}
for a principal $\G$-bundle $P\to Y$, a $\G$-equivariant function 
$u:\R\times P\to X$, and a $\G$-connection $A(s)+\Phi(s)\,ds$
on $\R\times P$.  This is analogous to the symplectic vortex 
equations~\cite{CGS,M2}, and similar equations were studied 
by Haydys~\cite{H1}. The Fueter equation in dimension 
four corresponds to $\G=1$, the Seiberg--Witten equations
to $X=\H$, $\G=S^1$ (Section~\ref{sec:SW}),
Taubes' generalized Seiberg--Witten equations
in~\cite{T} to $\G=S^1$, the Pidstrigatch--Tyurin 
equations~\cite{P,PT} to $X=\H$, $\G=\Sp(1)$, 
the instanton Floer equation to $X=\{\mathrm{point}\}$, 
and the Donaldson--Thomas $\G_2$-instantons on 
${Y=M\times\Sigma}$ ($\Sigma$ a hyperk\"ahler surface) appear 
when $X$ is taken to be the space of con\-nec\-tions on $\Sigma$
and $\G$ the group of gauge transformations 
(Section~\ref{sec:DT}).  

Multiplying the right hand side in~\eqref{eq:Gfueter} by $\eps^{-2}$ 
and taking the limit ${\eps\to0}$, one finds that equation~\eqref{eq:Gfueter} 
degenerates formally to the standard Fueter equation for functions 
with values in the hyperk\"ahler quotient $X\dslash\G=\mu^{-1}(0)/\G$. 
This is reminiscent of the Atiyah--Floer conjecture~\cite{DT}.
In~\cite{Wa1,Wa2} Walpuski carried out the adiabatic 
limit analysis to prove existence of $\G_2$-instantons.

Appendix~\ref{app:DFF} contains a proof of Gromov's theorem
that the inclusion of the space of global divergence-free frames
into the space of all frames is a homotopy equivalence.
Appendix~\ref{app:RC} discusses some functional analytic
background about self-adjoint Fredholm operators that
is needed in Section~\ref{sec:FUETER}. 

%%%%%%%%%%%%%%%%%%%%%%%%%%%%%%%%%%%%%%%%%%%%%% 
%%%%%%%%%%%%%%%%%%%%%%%%%%%%%%%%%%%%%%%%%%%%%% 
%%%%%%%%%%%%%%%% Section 2 %%%%%%%%%%%%%%%%%%% 
%%%%%%%%%%%%%%%%%%%%%%%%%%%%%%%%%%%%%%%%%%%%%% 
%%%%%%%%%%%%%%%%%%%%%%%%%%%%%%%%%%%%%%%%%%%%%% 
   
\section{The Fueter equation} \label{sec:FUETER}  

Let $M$ be a closed oriented three-manifold 
and $\dvol_M\in\Om^3(M)$ be a positive volume form.
Then $M$ is parallelizable and a theorem 
of Gromov~\cite{G} asserts that every frame of $TM$ 
can be deformed to a divergence-free frame.  
The space of positive divergence-free frames will be denoted by 
\begin{equation}\label{eq:frame}
\sV := \left\{(v_1,v_2,v_3)\in\Vect(M)^3\,\Bigg|\,
\begin{array}{l}
\cL_{v_i}\dvol_M=0\mbox{ for }i=1,2,3\\
\mbox{and }\dvol_M(v_1,v_2,v_3)>0\mbox{ on }M
\end{array}
\right\}.
\end{equation}
Thus $\sV$ is a subset of the space $\sF$ of positive frames. 
Formally, there is an analogy between the inclusion $\sV\hookrightarrow\sF$ 
and the inclusion of the space of symplectic forms into the space 
of all nondegenerate $2$-forms.  However, in contrast to the symplectic 
setting, the h-principle rules and the inclusion of $\sV$ into $\sF$ 
is a homotopy equivalence (see Theorem~\ref{thm:gromov}).  

Associated to every divergence-free frame ${v=(v_1,v_2,v_3)}$ 
is a self-adjoint Fredholm operator $\dd_v$ defined as follows.
Let $\H$ denote the space of quaternions
and write the elements of $\H$ in the form
$x=x_0+\i x_1+\j x_2+\k x_3$ with $x_0,x_1,x_2,x_3\in\R$.
Define the complex structures $I_1,I_2,I_3$ and $J_1,J_2,J_3$ 
on $\H$ as right and left multiplication by $\i,\j,\k$, 
respectively.  Thus, for $x\in\H$, 
\begin{equation}\label{eq:IJ}
\begin{split}
J_1x&:=\i x,\qquad\;\; J_2x:=\j x,\qquad\;\;\, J_3x:=\k x,\\
I_1x&:=-x\i,\qquad I_2x:=-x\j,\qquad I_3x:=-x\k.
\end{split}
\end{equation}
Thus the complex structures $J_1,J_2,J_3$ (respectively $I_1,I_2,I_3$)
satisfy the Clifford relations. Given a divergence-free frame
$v=(v_1,v_2,v_3)\in\sV$ define the linear first order 
differential operator 
$
\dd_v:\Om^0(M,\H)\to\Om^0(M,\H)
$
by 
\begin{equation}\label{eq:fueter}
\dd_vf:=J_1\p_{v_1}f+J_2\p_{v_2}f+J_3\p_{v_3}f.
\end{equation}
This is the {\bf Fueter operator}~\cite{Fu1,Fu2}.
It commutes with $I_i$ for $i=1,2,3$ and hence is equivariant 
under the right action of the quaternions on $\Om^0(M,\H)$. 
The divergence-free condition asserts that $\dd_v$ is symmetric 
with respect to the $L^2$ inner product on $\Om^0(M,\H)$ determined 
by the volume form and the standard inner product on $\H$. 
Denote this inner product by
$$
\Inner{f}{g}_{L^2} := \int_M\inner{f}{g}\,\dvol_M.
$$
The notation~\eqref{eq:fueter} extends to any triple of 
divergence-free vector fields and still defines a 
symmetric operator. It is self-adjoint, 
whenever $v$ is also a frame.

\begin{lemma}\label{le:elliptic}
Let $v=(v_1,v_2,v_3)\in\sV$.  Then 
$\dd_v:W^{1,2}(M,\H)\to L^2(M,\H)$
is a symmetric operator and,
for every $f\in L^2(M,\H)$,
\begin{equation}\label{eq:sadd}
\sup_{0\ne g\in\Om^0(M,\H)}
\frac{\Abs{\Inner{f}{\dd_vg}_{L^2}}}
{\Norm{g}_{L^2}}<\infty
\quad\iff\quad
f\in W^{1,2}(M,\H).
\end{equation}
\end{lemma}

\begin{proof}
That the operator is symmetric is obvious.
To prove~\eqref{eq:sadd}, define the divergence-free 
vector fields $w_1,w_2,w_3\in\Vect(M)$ 
by
\begin{equation}\label{eq:w}
w_1:=[v_2,v_3],\qquad w_2:=[v_3,v_1],\qquad w_3:=[v_1,v_2].
\end{equation}
(Here and throughout I use the sign convention 
$
\cL_{[u,v]}=-[\cL_u,\cL_v]
$ 
for the Lie bracket of two vector fields $u,v\in\Vect(M)$.)
Then 
\begin{equation}\label{eq:dd2}
\dd_v\dd_v = -\sL_v-\dd_w,\qquad
\sL_v:=\sum_{i=1}^3\p_{v_i}\p_{v_i}.
\end{equation}
Hence, taking $g=\dd_vh$ in~\eqref{eq:sadd}, we find
$$
\sup_{g\ne 0}
\frac{\Abs{\Inner{f}{\dd_vg}_{L^2}}}
{\Norm{g}_{L^2}}<\infty
\quad\implies\quad
\sup_{h\ne 0}
\frac{\Abs{\Inner{f}{\sL_vh}_{L^2}}}
{\Norm{h}_{W^{1,2}}}<\infty.
$$
Now it follows from standard elliptic regularity for the
second order operator $\sL_v$ that $f\in W^{1,2}(M,\H)$.  
This proves Lemma~\ref{le:elliptic}.
\end{proof}

%%%DAS - this is for version A
%It follows from Lemma~\ref{le:elliptic}
%and standard results in functional analysis 
%that $\dd_v:W^{1,2}(M,\H)\to L^2(M,\H)$ is a self-adjoint 
%index zero Fredholm operator for every $v\in\sV$. 
%%%
%%%DAS - this is for version AB
By Lemma~\ref{le:elliptic}, the operator 
$\dd_v:W^{1,2}(M,\H)\to L^2(M,\H)$ satisfies condition~(i) 
in Lemma~\ref{le:fredholm} and hence is a self-adjoint index 
zero  Fredholm operator for every $v\in\sV$. 
%%%
Its kernel consists of smooth functions, by 
elliptic regularity, and contains the constant functions. 
The dimension of the kernel is divisible by four.

\begin{definition}
A divergence-free frame $v\in\sV$ 
is called {\bf regular} if every solution
${f:M\to\H}$ of the linear Fueter equation
$\dd_vf=0$ is constant. Otherwise $v$ is called {\bf singular}.
The set of regular (respectively singular) divergence-free frames 
is denoted by $\sV^\reg$ (respectively $\sV^\sing$).
\end{definition}

\begin{lemma}\label{le:regular}
The set $\sV^\reg$ is open and dense in $\sV$.
\end{lemma}

\begin{proof}
Fix a divergence-free frame $v\in\sV$ and 
let $w=(w_1,w_2,w_3)$ be as in~\eqref{eq:w}.
Denote by $L^2_0(M,\H)\subset L^2(M,\H)$ 
and $W^{1,2}_0(M,\H)\subset W^{1,2}(M,\H)$ 
the spaces of functions with mean value zero
and consider the operator family 
$$
D(s) := \dd_{v+sw}:W^{1,2}_0(M,\H)\to L^2_0(M,\H).
$$
The path $s\mapsto D(s)$ is differentiable with 
derivative $\dot D(0)=\dd_w$.
By~\eqref{eq:dd2}, 
$$
\int_M\inner{f}{\dd_wg}\dvol_M
=\int_M\sum_{i=1}^3\inner{\p_{v_i}f}{\p_{v_i}g}\,\dvol_M
$$
for all $f,g\in W^{1,2}_0(M,\H)$ with $\dd_vf=\dd_vg=0$.
Hence the path $s\mapsto D(s)$ has a positive definite 
crossing form at $s=0$. Hence, 
%%%DAS - this is for version A
%by~\cite[Lemma~4.7]{RS2},
%%%
%%%DAS - this is for version AB
by Lemma~\ref{le:RC}, 
%%%
the operator $\dd_{v+sw}:W^{1,2}_0(M,\H)\to L^2_0(M,\H)$ 
is bijective for $s\ne 0$ sufficiently small 
and so $v+sw\in\sV^\reg$ for small $s\ne 0$.
Thus $\sV^\reg$ is dense in $\sV$.  
That $\sV^\reg$ is an open subset of $\sV$ is obvious. 
\end{proof}

The set of divergence-free frames decomposes 
as the disjoint union
\begin{equation}\label{eq:Vk}
\sV = \bigcup_{k=0}^\infty\sV_k,\qquad
\sV_k:=\left\{v\in\sV\,|\,\dim(\ker\,\dd_v) = 4(k+1)\right\}.
\end{equation}
In this notation $\sV^\reg = \sV_0$ and
$\sV^\sing = \bigcup_{k=1}^\infty\sV_k$.
By Lemma~\ref{le:regular}, $\sV_0$ is a dense open subset of $\sV$.

\begin{lemma}\label{le:V1}
$\sV_1$ is a codimension one Fr\'echet submanifold of $\sV$.
\end{lemma}

\begin{proof}
Let $\Om^0_0(M,\H)$ be the space of smooth functions $f:M\to\H$
with mean value zero and define
$$
\sB := \left\{(v,f)\in(\sV_0\cup\sV_1)\times\Om^0_0(M,\H)\,\bigg|\,
\int_M\Abs{f}^2\,\dvol_M=1\right\}.
$$
Since $\sV_0\cup\sV_1$ is an open subset of the Fr\'echet space 
of triples of divergence-free vector fields, 
$\sB$ is a Fr\'echet manifold with tangent spaces 
$$
T_{(v,f)}\sB = \left\{(\vhat,\fhat)\in\Vect(M)^3\times\Om^0_0(M,\H)\,\Bigg|\,
\begin{array}{l}
\cL_{\vhat_i}\dvol_M=0,\\
\int_M\inner{f}{\fhat}\,\dvol_M=0
\end{array}
\right\}.
$$
This space carries a free action of the Lie group $\Sp(1)$ 
(the unit quaternions) by $x_*(v,f):=(v,x_0f+\sum_{i=1}^3x_iI_if)$
for $x=x_0+\i x_1+\j x_2+\k x_3\in\Sp(1)$.  So does the total space 
of the vector bundle $\sE\to\sB$ with fibers
$$
\sE_{(v,f)} := \left\{h\in\Om^0_0(M,\H)\,\bigg|\,
\int_M\inner{h}{I_if}\,dvol_M=0\mbox{ for }i=1,2,3\right\}.
$$
This bundle has a natural $\Sp(1)$-equivariant section 
$$
\sS:\sB\to\sE,\qquad
\sS(v,f) := \dd_vf.
$$
The intrinsic differential of $\sS$ at a zero $(v,f)\in\sB$ 
is the linear operator
$$
\sD_{(v,f)}:T_{(v,f)}\sB\to\sE_{(v,f)},\qquad
\sD_{(v,f)}(\vhat,\fhat) := \dd_v\fhat +\dd_{\vhat}f.
$$
The kernel of the operator $\dd_v:\Om^0_0(M,\H)\to\Om^0_0(M,\H)$
is equal to its cokernel, has dimension four, and is spanned by 
$f,I_1f,I_2f,I_3f$ whenever $(v,f)\in\sB$ and $\dd_vf=0$. 
The summand $\fhat\mapsto\dd_v\fhat$ in $\sD_{(v,f)}$ 
is restricted to a codimension one subspace of $\Om^0_0(M,\H)$ 
(the $L^2$ orthgonal complement of $f$) while the target
is restricted to the codimension three subspace $\sE_{(v,f)}$ 
(the $L^2$ orthogonal complement of $I_1f,I_2f,I_3f$).
Thus its kernel has dimension three, its cokernel has dimension one,
and so its Fredholm index is two. Hence the projection
\begin{equation}\label{eq:two}
\left\{(\vhat,\fhat)\in T_{(v,f)}\sB\,|\,
\dd_v\fhat +\dd_{\vhat}f=0\right\}\to T_v\sV:
(\vhat,\fhat)\mapsto\vhat
\end{equation}
is a Fredholm operator of Fredholm index two. (See~\cite[Lemma~A.3.6]{MS}.)

Let $(v,f)\in\sB$ such that $\dd_vf=0$.  Then $\sD_{(v,f)}$ is surjective.
To see this, observe that $\sD_{(v,f)}:T_{(v,f)}\sB\to\sE_{(v,f)}$
has a closed image, by standard elliptic theory, and hence
it suffices to prove that the image is dense.  
Thus let $h\in L^2(M;\H)$ be a function 
with mean value zero, orthogonal to $I_1f,I_2f,I_3f$
(i.e.\ in the $L^2$ completion of $\sE_{(v,f)}$), and
$L^2$ orthogonal to the image of $\sD_{(v,f)}$.  Then
\begin{equation}\label{eq:h1}
\int_Mh\,\dvol_M=0,\qquad
\int_M\inner{h}{I_if}\dvol_M=0\;\;\;
\mbox{ for }i=1,2,3,
\end{equation}
\begin{equation}\label{eq:h2}
\int_M\inner{h}{\dd_v\fhat}\dvol_M=0\;\;\;
\mbox{ for }\fhat\in\Om^0_0(M,\H)
\mbox{ with }\int_M\inner{f}{\fhat}\dvol_M=0,
\end{equation}
\begin{equation}\label{eq:h3}
\int_M\inner{h}{\dd_\vhat f}\dvol_M=0
\;\;\;\mbox{ for }\vhat\in T_v\sV. 
\end{equation}
It follows from~\eqref{eq:h2} and elliptic regularity that $h:M\to\H$
is smooth and that $\dd_vh-h_0\in\R f$ for some element $h_0\in\H$.
Since $\dd_vf=0$ it follows that $\dd_v\dd_vh=0$ and hence $\dd_vh=0$. 
(Take the $L^2$-inner product with $h$ and integrate by parts.) 
Since the kernel of $\dd_v:\Om^0_0(M,\H)\to\Om^0_0(M,\H)$
is spanned by $f,I_1f,I_2f,I_3f$ it follows from~\eqref{eq:h1} 
that $h=\lambda f$ for some $\lambda\in\R$.
Hence it follows from~\eqref{eq:h3}, with $\vhat=w$ given 
by~\eqref{eq:w}, and from~\eqref{eq:dd2} that
$$
0 = \int_M\inner{h}{\dd_wf}\dvol_M
= \int_M\sum_{i=1}^3\inner{\p_{v_i}h}{\p_{v_i}f}\,\dvol_M
= \lambda\int_M\Abs{df}^2\,\dvol_M.
$$
Since $f$ is nonconstant, this implies $\lambda=0$ and hence $h=0$.

This shows that the operator $\sD_{(v,f)}:T_{(v,f)}\sB\to\sE_{(v,f)}$
is surjective for all $(v,f)\in\sB$ with $\dd_vf=0$.
Via Sobolev completion and the infinite-dimensional 
implicit function theorem, it follows that the set 
$$
\sP:=\bigl\{(v,f)\in\sB\,\big|\,\dd_vf=0\bigr\}
$$
is a Fr\'echet submanifold of $\sB$ with tangent spaces
$$
T_{(v,f)}\sP = \bigl\{(\vhat,\fhat)\in T_{(v,f)}\sB\,\big|\,
\dd_v\fhat +\dd_{\vhat}f=0\bigr\}.
$$
The group $\Sp(1)$ acts freely on $\sP$ and the quotient
space $\sP/\Sp(1)$ is homeo\-morphic to $\sV_1$ 
via the projection $\pi:\sP\to\sV$ defined by $\pi(v,f):=v$.
The derivative of $\pi$ is the linear operator 
$d\pi(v,f):T_{(v,f)}\sP\to T_v\sV$ given by $d\pi(v,f)(\vhat,\fhat)=\vhat$. 
This is a Fredholm operator of index two (equation~\eqref{eq:two}) 
and it has a three-dimensional kernel.  
Hence its cokernel has dimension one.  This implies, again via 
suitable Sobolev completions, that the map $\pi:\sP\to\sV$ 
descends to an injective immersion 
$\iota:\sP/\Sp(1)\to\sV$
with image $\sV_1$ and derivative of Fredholm index $-1$.  
The immersion $\iota$ is obviously proper
(compact subsets of $\sV_1$ have compact preimages 
in $\sP/\Sp(1)$) and hence $\iota$ is an embedding.  
This proves Lemma~\ref{le:V1}.
\end{proof}

\begin{conjecture}\label{con:strat}
$\sV_k$ is a Fr\'echet submanifold
of $\sV$ of codimension $2k^2-k$.
\end{conjecture}

Conjecture~\ref{con:strat} is motivated by an analogous result 
for quaterni\-onic-hermitian matrices and by Lemma~\ref{le:V1}.
If the conjecture is true, then every path $s\mapsto v^s$ in~$\sV$
with endpoints in $\sV_0$ is homotopic, with fixed endpoints, 
to a path that is transverse to $\sV_1$ and misses $\sV_k$ for $k>1$.
The resulting path $s\mapsto\dd_{v^s}$ of self-adjoint Fredholm 
operators then has only regular crossings and its spectral flow 
is the intersection number with~$\sV_1$.

\begin{remark}[{\bf Spectral flow}]\label{rmk:specflow}\rm
A loop of divergence-free frames with non\-zero spectral flow, 
if it exists, has infinite order in~$\pi_1(\sV)$. The existence 
of such a loop would prove that $\sV^\sing\ne\emptyset$.  
If the fundamental group of a connected component
$\sV'\subset\sV$ is finite and 
$
\sV^\sing\cap\sV'\ne\emptyset,
$
then the spectral flow of a path with endpoints in $\sV^\reg\cap\sV'$ 
depends only on the endpoints, so $\sV^\reg\cap\sV'$ is disconnected.
(See Example~\ref{ex:S3sing} below for $M=S^3$.)
\end{remark}

The next lemma relates the topology of $\sV$ to the topology
of the group of gauge transformations 
\begin{equation}\label{eq:G}
\sG := \Map(M,\SO(3)).
\end{equation}
The identity component of $\sG$ will be denoted by 
$$
\sG_0:=\left\{g:M\to\SO(3)\,\big|\,g
\mbox{ lifts to a degree zero map }\tg:M\to\Sp(1)\right\}
$$
and the group of based gauge transformations,
associated to $p^*\in M$, by
$$
\sG^* := \left\{g:M\to\SO(3)\,|\, g(p^*)=\one\right\}.
$$
Their intersection is the group 
\begin{equation}\label{eq:G*)}
\sG_0^* := \sG_0\cap\sG^*
\cong\left\{\tg:M\to\Sp(1)\,\big|\,\deg(\tg)=0,\,\tg(p^*)=1\right\}.
\end{equation}

\begin{lemma}\label{le:topology}
{\bf (i)}
$\sV$ is homotopy equivalent to $\sG$ and
each connected component of $\sV$ is homotopy 
equivalent to $\sG^*_0\times\SO(3)$.

\smallskip\noindent{\bf (ii)}
There is a short exact sequence
\begin{equation}\label{eq:Vpi0}
\xymatrix    
@C=30pt    
@R=20pt    
{ 
& & \pi_0(\sV)\ar[d]_{\cong} \\
0\ar[r] 
& \Z\ar[r]
& \pi_0(\sG) \ar[r]
& H^1(M;\Z_2)\ar[r]
& 0 
},
\end{equation}
where the homomorphism $\Z\to\pi_0(\sG)$ assigns to an integer $d$
the homotopy class of maps $M\to\SO(3)$ that lift to degree-$d$
maps $M\to\Sp(1)$, and the homomorphism 
$\pi_0(\sG)=\pi_0(\sG^*)\to\Hom(\pi_1(M,p^*),\Z_2)=H^1(M;\Z_2)$
assigns to a based gauge transformation $g:M\to\SO(3)$
the induced homomorphism $g_*:\pi_1(M,p^*)\to\pi_1(\SO(3),\one)=\Z_2$.

\smallskip\noindent{\bf (iii)}
Let $\sV'$ be a connected component of $\sV$.
Then 
$$
\pi_1(\sV') \cong \pi_1(\sG_0^*)\times\Z_2.
$$
If $M=S^3$ then $\pi_1(\sV') \cong\Z_2\times\Z_2$. 

\smallskip\noindent{\bf (iv)}
Let $M=S^3=\Sp(1)\subset\H$ be the unit sphere in the quaternions
and denote by $\sV'\subset\sV$ the connected component of the 
standard frame of $TM$. Then the inclusion
\begin{equation}\label{eq:SO3V}
\SO(3)\to\sV':A\mapsto v_A=(v_{A,1},v_{A,2},v_{A,3}),
\end{equation}
which assigns to a matrix $A=(a_{ij})_{1\le i,j\le3}\in\SO(3)$ 
the divergence-free frame defined by 
$v_{A,i}(y) := a_{1i}\i y + a_{2i}\j y + a_{3i}\k y$
for $y\in S^3$ and $i=1,2,3$, 
induces an isomorphism on rational homology. 
\end{lemma}

\begin{proof}
By Theorem~\ref{thm:gromov} the inclusion of~$\sV$ into the space 
$\sF$ of positive global frames is a homotopy equivalence.
So is the inclusion into $\sF$ of the space of positive global 
orthonormal frames (associated to a Riemannian metric). 
The latter is homeo\-morphic to the gauge group $\sG$.
Hence $\sV$ is homotopy equivalent to $\sG$.
An explicit homotopy equivalence from $\sV$ to $\sG$ can be
constructed by fixing a positive global frame of $TM$ and defining
\begin{equation}\label{eq:VG}
\sV\to\sG:v\mapsto g_v:=A_v(A_v^TA_v)^{-1/2}, 
\end{equation}
where $A_v:M\to\GL^+(\R^3)$ is the gauge transformation 
relating $v$ to the reference frame.  
Now the map
$$
\sG\to\sG^*\times\SO(3):g\mapsto (gg(p^*)^{-1},g(p^*))
$$
is a homeomorphism and restricts to a homeomorphism
$\sG_0\cong\sG^*_0\times\SO(3)$.  This proves~(i).

Next observe that a based gauge transformation 
$g:M\to\SO(3)$ lifts to a map $\tg:M\to\Sp(1)$ 
if and only if the induced map 
$$
g_*:\pi_1(M,p^*)\to\pi_1(\SO(3),\one)=\Z_2
$$ 
on fundamental groups is trivial.
Hence the kernel of the map
$$
\sG^*\to \Hom(\pi_1(M,p^*),\Z_2)=H^1(M,\Z_2):g\mapsto g_*
$$
is isomorphic to the subgroup of all based
gauge transformations that lift to maps
$M\to S^3$. Hence exactness at $\Z$ 
and exactness at $\pi_0(\sG)\cong\pi_0(\sG^*)$ 
follow from the Hopf degree theorem.
Exactness at $H^1(M;\Z_2)$ follows by considering
the subgroup of gauge transformations with values in the
standard circle in $\SO(3)$.  This proves~(ii).

It follows immediately from~(i) that
$
\pi_1(\sV') \cong \pi_1(\sG_0^*)\times\Z_2.
$
If $M=S^3$ then 
$
\pi_1(\sG_0^*)\cong\pi_4(S^3)\cong\Z_2
$
and hence $\pi_1(\sV') \cong\Z_2\times\Z_2$.  
This proves~(iii). 

Assertion~(iv) follows from~(i) 
and a result of Donaldson and Kronheimer
(Lemma~5.1.14 in~\cite{DK}), which asserts that
the group 
$
\sG_0^*\cong(\Om^3S^3)_0
$
has the rational homology of a point.  
This proves Lemma~\ref{le:topology}.
\end{proof}

\begin{remark}[{\bf K-theory}]\label{rmk:Ktheory}\rm
The positive divergence-free frames $v\in\sV$ parametrize
the {\it universal family} of (self-adjoint) Fueter 
operators $\dd_v$. This family determines a K-theory class 
$$
\INDEX(\dd)\in K^{-1}(\sV)=K^0(S^1\times\sV)
$$
(see Atiyah--Patodi--Singer~\cite{APS}).
Its Chern character determines an odd-dimensional  
cohomology class $\ch(\INDEX(\dd))\in H^\odd(\sV;\Q)$
(given by the spectral flow in degree one).
When $M=S^3$ and $\sV'$ is the connected component 
of the standard frame, it follows from 
Lemma~\ref{le:topology}~(iv) that the restriction 
of the class $\ch(\INDEX(\dd))$ to $\sV'$ is 
determined by its pullback under the embedding 
$\SO(3)\to\sV'$ in~\eqref{eq:SO3V}.
Since the image of this embedding is contained in $\sV_0$ 
(see Example~\ref{ex:S3} below), the dimension of the kernel 
of $\dd_v$ is constant along this embedding and hence its class 
in $K^{-1}(\SO(3))$ is trivial (see Ebert~\cite[Theorem~4.2.1]{E}).  
Hence the Chern character of the K-theory class 
$\INDEX(\dd)\in K^{-1}(\sV')$ vanishes in $H^\odd(\sV';\Q)$,
when $M=S^3$, and thus the K-theory class itself vanishes
in $K^{-1}(\sV')$ modulo torsion. 
\end{remark}

\begin{remark}[{\bf Dual frame}]\label{rmk:dualframe}\rm
Let $v\in\sV$ and denote by $\alpha_1,\alpha_2,\alpha_3\in\Om^1(M)$ 
the dual frame so that $\alpha_i(v_j)=\delta_{ij}$.
Define a metric on $M$ by
\begin{equation}\label{eq:metric}
\inner{u}{v} := \sum_{i=1}^3\alpha_i(u)\alpha_i(v),\qquad
u,v\in T_yM.
\end{equation}
Then the vector fields $v_1,v_2,v_3$ form an 
orthonormal frame and the volume form is 
$\alpha_1\wedge\alpha_2\wedge\alpha_3$.
It does not, in general, agree with $\dvol_M$.  
They agree up to a constant factor if and only if the $2$-forms
$\alpha_j\wedge\alpha_k$ are closed.
To see this, define $\lambda:=\dvol_M(v_1,v_2,v_3)$.
Then $\iota(v_i)\dvol_M=\lambda\alpha_j\wedge\alpha_k$
for every cyclic permutation $i,j,k$ of $1,2,3$.
If $\alpha_j\wedge\alpha_k$ is closed for $j\ne k$
it follows that $d\lambda\wedge\alpha_j\wedge\alpha_k=0$,
hence $\p_{v_i}\lambda=0$ for all $i$, and hence $\lambda$ 
is constant. 
\end{remark}

\begin{remark}[{\bf Laplace--Beltrami operator}]\label{rmk:LB}\rm
The pointwise norm of the differential $df$ of a function
$f:M\to\H$ with respect to the metric~\eqref{eq:metric} is 
given by $\Abs{df}^2=\sum_i\Abs{\p_{v_i}f}^2$.  If the function 
$\lambda := \dvol_M(v_1,v_2,v_3)$ is constant, then~$\sL_v$ 
is the Laplace--Beltrami operator of the metric~\eqref{eq:metric}.  
Otherwise it is the composition of $d$ and its formal adjoint with 
respect to the $L^2$ inner products on functions and $1$-forms
associated to the pointwise inner products of the 
metric~\eqref{eq:metric} and the volume form $\dvol_M$.
It is then related to the Laplace--Beltrami operator by the formula
$$
\sL_vf = -\frac{1}{\lambda}d^*(\lambda df) = -d^*df 
+ \frac{1}{\lambda}\sum_i(\p_{v_i}\lambda)\p_{v_i}f
$$
for $f\in\Om^0(M,\H)$. 
\end{remark}

\begin{definition}
A divergence-free frame $(v_1,v_2,v_3)\in\sV(M,\dvol_M)$
is called {\bf normal} if $\dvol_M(v_1,v_2,v_3)=1$.
\end{definition}

\begin{lemma}\label{le:normal}
Every positive regular (respectively singular) diver\-gence-free frame
$v\in\sV(M,\dvol_M)$ can be deformed through regular (respectively singular)
divergence-free frames to a normal regular (respectively singular)
diver\-gence-free frame.
\end{lemma}

\begin{proof}
Let $v\in\sV^\reg(M,\dvol_M)$ and 
$\lambda:=\dvol_M(v_1,v_2,v_3)$. Define
$$
\rho_t := c_t\lambda^{t/2}\dvol_M,\qquad 
v_{i,t} := c_t^{-1/3}\lambda^{-t/2}v_i,\qquad 
c_t := \frac{\int_M\dvol_M}{\int_M\lambda^{t/2}\dvol_M}.
$$
Then $v_t:=(v_{1,t},v_{2,t},v_{3,t})\in\sV(M,\rho_t)$ for every $t$. 
Second, $v_t$ is equal to $v$ for $t=0$ and is a normal frame for $t=1$.
Third, $v_t$ is regular for every $t$, 
because $\ker\dd_{v_t}=\ker\dd_v$.
Fourth $\rho_0=\dvol_M$ and $\int_M\rho_t=\int_M\dvol_M$ 
for all~$t$.  By Moser isotopy, 
choose a smooth isotopy $[0,1]\to\Diff(M):t\mapsto\phi_t$ 
such that $\phi_0=\id$ and $\phi_t^*\rho_t=\dvol_M$ 
for all $t$. The required isotopy of regular frames 
is $\phi_t^*v_t\in\sV^\reg(M,\dvol_M)$, $0\le t\le 1$. 
Namely, the frame $\phi_t^*v_t$ is regular for 
every $t$ because 
$$
\ker\,\dd_{\phi_t^*v_t}
= \phi_t^*\ker\,\dd_{v_t}
= \phi_t^*\ker\,\dd_v
= 0.
$$
This proves Lemma~\ref{le:normal}.
\end{proof}

%%%%%%%%%%%%%%%%%%%%%%%%%%%%%%%%%%%%%%%%%%%%%%    
%%%%%%%%%%%%%%%%%%%%%%%%%%%%%%%%%%%%%%%%%%%%%%    
%%%%%%%%%%%%%%%% Section 3 %%%%%%%%%%%%%%%%%%%    
%%%%%%%%%%%%%%%%%%%%%%%%%%%%%%%%%%%%%%%%%%%%%%    
%%%%%%%%%%%%%%%%%%%%%%%%%%%%%%%%%%%%%%%%%%%%%%    
    
\section{Examples} \label{sec:EX}  

Fix a divergence-free frame
$v\in\sV(M,\dvol_M)$ and let $\alpha_1,\alpha_2,\alpha_3\in\Om^1(M)$
be the dual frame (see Remark~\ref{rmk:dualframe}). 
For $i=1,2,3$ define $\om_i\in\Om^2(\H)$ by
$\om_i=dx_0\wedge dx_i+dx_j\wedge dx_k$,
where $i,j,k$ is a cyclic permutation of $1,2,3$.
Let $f:M\to\H$ be a smooth map and abbreviate
$\Abs{df}^2:=\sum_i\Abs{\p_{v_i}f}^2$. Then 
there is an {\bf energy identity}
\begin{equation}\label{eq:ENERGY}
\frac12\int_M\Abs{df}^2\,\dvol_M
= \frac12\int_M\Abs{\dd_vf}^2\,\dvol_M
- \sum_i\int_M\alpha_i\wedge f^*\om_i.
\end{equation}
This continues to hold for maps from $M$
to any hyperk\"ahler manifold $X$. 

\begin{example}\label{ex:T3}\rm
Consider the $3$-torus $M=\T^3=\R^3/\Z^3$.
Every linearly independent triple of 
constant vector fields $v_1,v_2,v_3$
is a regular divergence-free frame.
Namely, the dual frame consists of closed $1$-forms,
hence the second term on the right in~\eqref{eq:ENERGY} 
vanishes for functions $f:M\to\H$, and hence
every solution of the Fueter equation is constant. 
\end{example}

\begin{example}\label{ex:S3}\rm
Consider the $3$-sphere $M=S^3\subset\H$
with the volume form
$
\dvol_M = y_0dy_1dy_2dy_3
-y_1dy_0dy_2dy_3
-y_2dy_0dy_3dy_1
-y_3dy_0dy_1dy_2
$
and the frame
$$
v_1(y) = \i y,\qquad 
v_2(y) = \j y,\qquad
v_3(y) = \k y.
$$
The $v_i$ are divergence-free 
and $\dvol_M(v_1,v_2,v_3)=1$.
The dual frame is 
\begin{equation*}
\begin{split}
\alpha_1 &= y_0dy_1-y_1dy_0+y_2dy_3-y_3dy_2, \\
\alpha_2 &= y_0dy_2-y_2dy_0+y_3dy_1-y_1dy_3, \\
\alpha_3 &= y_0dy_3-y_3dy_0+y_1dy_2-y_2dy_1.
\end{split}
\end{equation*}
Note that $[v_j,v_k]=2v_i$ and 
$d\alpha_i=2\alpha_j\wedge\alpha_k=2\iota(v_i)\dvol_M$ 
for every cyclic permutation $i,j,k$ of $1,2,3$.  
The energy identity~\eqref{eq:ENERGY} has the form
$$
\frac12\int_M\Abs{df}^2\dvol_M
=\frac12\int_M\Abs{\dd_vf}^2\dvol_M 
+ \int_M\inner{f}{\dd_vf}\dvol_M.
$$
(See~\eqref{eq:ACTION} and~\eqref{eq:action} below.)
Hence every solution $f:M\to\H$ of
the Fueter equation is constant, so $v$ is a 
normal regular divergence-free frame.
\end{example}

\begin{example}\label{ex:S3sing}\rm
Consider the $3$-sphere $M=S^3\subset\H$. The frame
$$
v_1(y) = 2^{2/3}\i y,\qquad 
v_2(y) = -2^{-1/3}\j y,\qquad
v_3(y) = -2^{-1/3}\k y,
$$
is a normal singular divergence-free frame on the $3$-sphere
for the standard volume form.  The obvious inclusion
$f:S^3\to\H$ is a nonconstant solution of the Fueter equation.

This example shows that $\sV^\sing\cap\sV'\ne\emptyset$, where 
$\sV'\subset\sV$  denotes the connected component of the standard 
frame on $S^3$. Since $\pi_1(\sV')\cong\Z_2\times\Z_2$ is a finite 
group (see Lemma~\ref{le:topology}~(iii)) it follows that 
$\sV^\reg\cap\sV'$ is disconnected (see Remark~\ref{rmk:specflow}).  
\end{example}

\begin{example}\label{ex:S1S2}\rm
Consider the product $M:=S^1\times S^2$, where 
$S^1$ is the unit circle with coordinate $e^{\i\theta}$ 
and $S^2\subset\R^3$ is the unit sphere with 
coordinates $y_1,y_2,y_3$. The standard volume form is 
$\dvol_M:=d\theta\wedge\dvol_{S^2}$, where
$
\dvol_{S^2} := y_1dy_2dy_3+y_2dy_3dy_1+y_3dy_1dy_2.
$
Define $v_i,w_i,\alpha_i,\beta_i$ by
\begin{equation*}
\begin{split}
v_1 := y_1\frac{\p}{\p\theta} 
+ y_2\frac{\p}{\p y_3} - y_3\frac{\p}{\p y_2},\qquad
&\alpha_1 := y_1d\theta + y_2dy_3-y_3dy_2,\\
w_1 := 2y_1\frac{\p}{\p\theta} 
+ y_2\frac{\p}{\p y_3} - y_3\frac{\p}{\p y_2},\qquad
&\beta_1 := y_1d\theta + \tfrac12\left(y_2dy_3-y_3dy_2\right),
\end{split}
\end{equation*}
for $i=1$ and by cyclic permutation for $i=2,3$. 
Then $v_1,v_2,v_3$ is a divergence-free orthonormal frame 
and $\alpha_1,\alpha_2,\alpha_3$ is the dual frame. Moreover,
$$
w_i=[v_j,v_k],\qquad
d\beta_i=\alpha_j\wedge\alpha_k=\iota(v_i)\dvol_M
$$ 
for every cyclic permutation $i,j,k$ of $1,2,3$. 

The energy identity~\eqref{eq:ENERGY} takes the form
\begin{equation}\label{eq:energyS1S2}
\frac12\int_M\Abs{df}^2\dvol_M = 
\frac12\int_M\Abs{\dd_vf}^2\dvol_M
+ \frac12\int_M\inner{f}{\dd_vf}\dvol_M + \widehat{\sA}(f),
\end{equation}
where 
$$
\widehat\sA(f) := -\sum_i\int_M
\widehat\alpha_i\wedge f^*\om_i,\qquad
\widehat\alpha_i 
:= \tfrac12\left(y_jdy_k-y_kdy_j\right),
$$
for every cyclic permutation $i,j,k$ of $1,2,3$.
For $y\in S^2$ define 
$$
\om_y := y_1\om_1+y_2\om_2+y_3\om_3\in\Om^2(\H).
$$
Then
\begin{equation}\label{eq:Ahat}
\widehat{\sA}(f) 
= \frac12\int_{S^2}\int_0^{2\pi}
\om_y(\p_\theta f,f)\,d\theta\,\dvol_{S^2}.
\end{equation}
(This discussion extends to maps with values in any hyperk\"ahler
manifold~$X$, if the second summand in~\eqref{eq:energyS1S2}
is replaced by $\sA(f):=-\sum_i\int_M\beta_i\wedge f^*\om_i$
and the integrand in~\eqref{eq:Ahat} by the symplectic action of 
the loop $\theta\mapsto f(e^{\i\theta},y)$ with respect to $\om_y$.)

The isoperimetric inequality asserts that
\begin{equation}\label{eq:isoperimetric}
\frac12\int_0^{2\pi}\om_y(\p_\theta f,f)\,d\theta
\le \frac12\int_0^{2\pi}\Abs{\p_\theta f}^2\,d\theta,\qquad
y\in S^2.
\end{equation}
This inequality is sharp (see~\cite[Section~4.4]{MS}). 
Now let 
$
f:S^1\times S^2\to\H
$
be a solution of the Fueter equation.  
By~\eqref{eq:energyS1S2}, \eqref{eq:Ahat}, 
and~\eqref{eq:isoperimetric} it satisfies 
\begin{equation}\label{eq:S1S2}
\frac12\int_M\Abs{df}^2\dvol_M
=\widehat{\sA}(f) \le \frac12\int_M\Abs{\p_\theta f}^2\dvol_M.
\end{equation}
Since 
$$
\frac{\p}{\p\theta} = \sum_iy_iv_i
$$ 
is a unit vector field, equality must hold 
in~\eqref{eq:S1S2} and 
$$
\abs{df}\equiv\Abs{\p_\theta f}.
$$ 
Hence $f$ is independent of the variable in $S^2$.  
Thus the derivative of $f$ has rank one everywhere and so $f$ is constant.  
This shows that $v_1,v_2,v_3$ is a normal regular divergence-free frame.
\end{example}

%%%%%%%%%%%%%%%%%%%%%%%%%%%%%%%%%%%%%%%%%%%%%% 
%%%%%%%%%%%%%%%%%%%%%%%%%%%%%%%%%%%%%%%%%%%%%% 
%%%%%%%%%%%%%%%% Section 4 %%%%%%%%%%%%%%%%%%% 
%%%%%%%%%%%%%%%%%%%%%%%%%%%%%%%%%%%%%%%%%%%%%% 
%%%%%%%%%%%%%%%%%%%%%%%%%%%%%%%%%%%%%%%%%%%%%% 

\section{Hyperk\"ahler Floer theory}\label{sec:HF} 

Fix a closed connected oriented three-manifold $M$, 
a volume form $\dvol_M$,
a divergence-free frame $v\in\sV(M,\dvol_M)$,
and a closed hyperk\"ahler mani\-fold~$X$ with symplectic forms 
$\om_1,\om_2,\om_3$ and complex structures $J_1,J_2,J_3$.
Denote the hyperk\"ahler metric on~$X$ by 
$\inner{\cdot}{\cdot}=\om_i(\cdot,J_i\cdot)$.
Consider the Fueter equation
with a {\bf Hamiltonian perturbation}.  It has the form
\begin{equation}\label{eq:fueterH}
J_1\p_{v_1}f + J_2\p_{v_2}f + J_3\p_{v_3}f = \nabla H(f)
\end{equation}
for a map $f:M\to X$.  The left hand side 
of~\eqref{eq:fueterH} will still be denoted by $\dd_vf$. 
The operator $\dd_v$ should now be thought 
of as a vector field on the infinite-dimensional space 
$
\cF:=\Cinf(M,X)
$
of all smooth maps from $M$ to~$X$. The perturbation 
is determined by a smooth function $H:M\times X\to\R$ 
and~$\nabla H$ is the gradient with respect to the variable in~$X$. 
A solution $f$ of~\eqref{eq:fueterH} is called 
{\bf nondegenerate} if the linearized operator 
of~\eqref{eq:fueterH} is bjective.
The next theorem is a hyperk\"ahler analogue of the 
Arnold conjecture for Hamiltonian systems on the torus 
as proved by Conley--Zehnder~\cite{CZ}. 

\begin{theorem}[\cite{HNS1,HNS2,GH1,GH2}]\label{thm:AC}
Assume $v\in\sV^\reg$ and $X$ is flat.  
If every contractible solution of~\eqref{eq:fueterH}
is nondegenerate then their number 
is bounded below by $\dim H^*(X;\Z_2)$.
In particular, \eqref{eq:fueterH} has a contractible
solution for every $H$. 
\end{theorem}

There are two proofs of Theorem~\ref{thm:AC}. 
One is due to Sonja Hohloch, Gregor Noetzel, 
and the present author. It is based on a hyperk\"ahler 
analogue of Floer theory and is carried out in~\cite{HNS1,HNS2} 
for the $3$-sphere and the $3$-torus. The second proof is 
due to Viktor Ginzburg and Doris Hein, is based on 
finite-dimensional reduction, and is carried out in full 
generality in~\cite{GH1,GH2}. Both proofs rely on the 
following fundamental estimate.

\begin{lemma}\label{le:estimate}
Assume $v\in\sV^\reg$ and $X$ is flat.  Then there is a constant~$c$
such that every contractible smooth map $f:M\to X$ satisfies 
\begin{equation}\label{eq:estimate}
\int_M\Abs{df}^2\,\dvol_M
\le c\int_M\Abs{\dd_vf}^2\,\dvol_M.
\end{equation} 
\end{lemma}

\begin{proof}
Since $X$ is flat, it is isomorphic to a quotient of a 
torus $\H^n/\Lambda$ by the free action of a finite group.
Hence every contractible map $f:M\to X$ lifts to a map
from $M$ to $\H^n$.  Thus it suffices to 
prove~\eqref{eq:estimate} for ${f:M\to\H}$.
By Lemma~\ref{le:elliptic}
%%%DAS - this is for version AB 
%%%DAS - deleted for version A
%%%
and Lemma~\ref{le:fredholm}, 
%%%
the operator ${\dd_v:W^{1,2}_0(M,\H)\to L^2_0(M,\H)}$ 
is Fredholm and has index zero.
Since $v\in\sV^\reg$, this operator is bijective.  
Hence, for functions $f:M\to\H$ with mean value zero, 
\eqref{eq:estimate} follows from the inverse operator theorem.  
Adding a constant to $f$ does not affect~\eqref{eq:estimate}. 
\end{proof}

Here is an outline of the Floer theory proof of Theorem~\ref{thm:AC}.
The space $\cF$ of maps from $M$ to $X$ carries a natural $1$-form
$\Psi_f:T_f\cF\to\R$ defined by 
\begin{equation}\label{eq:Psi}
\Psi_f(\hat f) = \sum_i\int_M\om_i(\p_{v_i}f,\hat f)\,\dvol_M
\end{equation}
for a vector field $\hat f\in T_f\cF = \Om^0(M,f^*TX)$ along $f$.
This $1$-form is closed because the vector fields $v_i$ 
are divergence-free.

\begin{remark}\label{rmk:action1}\rm
Assume that the $2$-forms $\iota(v_i)\dvol_M$ are exact
and choose $1$-forms $\beta_i\in\Om^1(M)$ 
such that $d\beta_i=\iota(v_i)\dvol_M$ for $i=1,2,3$.
Then the $1$-form $\Psi$ in~\eqref{eq:Psi}
is the differential of the {\bf hyperk\"ahler action functional}
\begin{equation}\label{eq:ACTION}
\sA(f) := - \sum_i\int_M\beta_i\wedge f^*\om_i,\qquad f\in\cF.
\end{equation}
\end{remark}

\begin{remark}\label{rmk:action2}\rm
Assume $X$ is flat and let $\cF_0\subset\cF$ be
the connected component of the constant maps.
Then the restriction of $\Psi$ to $\cF_0$ is the differential 
of the action functional $\sA=\sA_v:\cF_0\to\R$, defined by
\begin{equation}\label{eq:action}
\sA(f):=\frac12\int_M\inner{f}{\dd_vf}\,\dvol_M,\qquad
f\in\cF_0.
\end{equation}
To understand the right hand side, lift $f$ to a function 
with values in~$\H^n$ and observe that the integrand
is invariant under the action of the hyperk\"ahler 
isometry group of $\H^n$.
\end{remark}

Assume from now on that $X$ is flat.
The contractible solutions of~\eqref{eq:fueterH} 
are the critical points of the {\bf perturbed hyperk\"ahler action 
functional} $\sA_H=\sA_{v,H}:\cF_0\to\R$, defined by
$$
\sA_H(f) := \sA(f) - \int_MH(f)\,\dvol_M,\qquad f\in\cF_0.
$$
The gradient flow lines of $\sA_H$ are the solutions $u:\R\times M\to X$
of the perturbed four-dimensional Fueter equation
\begin{equation}\label{eq:floerH}
\p_su+J_1\p_{v_1}u+J_2\p_{v_2}u+J_3\p_{v_3}u = \nabla H(u),
\end{equation}
\begin{equation}\label{eq:limitH}
\lim_{s\to\infty}u(s,y) = f^\pm(y).
\end{equation}
Here $f^\pm\in\cF_0$ are solutions of~\eqref{eq:fueterH}.
The convergence in~\eqref{eq:limitH} is in the $\Cinf$ topology on $M$ 
and exponential in $s$.  The solutions of~\eqref{eq:floerH} 
and~\eqref{eq:limitH} satisfy the usual energy identity
\begin{equation}\label{eq:energyH}
\sE(u) 
:= \int_{-\infty}^\infty\int_M\Abs{\p_su}^2\dvol_M\,ds 
= \sA_H(f^-)-\sA_H(f^+).
\end{equation}
By Lemma~\ref{le:estimate}, the energy controls the $W^{1,2}$-norms
of the solutions of~\eqref{eq:floerH} on every compact subset of $\R\times M$.  
Since the leading term in~\eqref{eq:floerH} is a linear elliptic operator, 
this suffices for the standard regularity and compactness
arguments in symplectic Floer theory to extend to the present setting.  
For $M=S^3$ the proof is carried out in detail in~\cite[Section~3]{HNS2} 
and the arguments extend verbatim to general three-manifolds. 
The same holds for unique continuation and transversality  
in~\cite[Section~4]{HNS2}. An index formula involving 
the spectral flow shows that there is a function 
$
\mu_H:\mathrm{Crit}(\sA_H)\to\Z
$
such that the Fredholm index of the linearized operator of
equation~\eqref{eq:floerH} is equal to the difference 
$\mu_H(f^-)-\mu_H(f^+)$ for every solution $u$ 
of~\eqref{eq:floerH} and~\eqref{eq:limitH}
(see~\cite[Section~4]{HNS2}).  
The third ingredient in the analysis is a gluing result and  
it follows from a standard adaptation of Floer's gluing 
theorem~\cite{F1,F2,F3} to the hyperk\"ahler setting.  
The upshot is, that the contractible solutions 
of~\eqref{eq:fueterH} generate a Floer chain complex
$$
\CFHK_*(M,X,\tau_0;v,H) := \bigoplus_{f\in\mathrm{Crit}(\sA_H)}\Z_2f
$$
with $\Z_2$ coefficients. Here $\tau_0\in\pi_0(\cF)$ 
denotes the homotopy class of the constant maps. 
The Floer complex is graded by the index function $\mu_H$, 
and the boundary operator 
$
\p:\CFHK_k(M,X,\tau_0;v,H)\to\CFHK_{k-1}(M,X,\tau_0;v,H)
$
is defined by the mod two count of the solutions 
of~\eqref{eq:floerH} and~\eqref{eq:limitH} (modulo time shift) 
in the case $\mu_H(f^-)-\mu_H(f^+)=1$.  
The {\bf hyperk\"ahler Floer homology groups}
$$
\HFHK_*(M,X,\tau_0;v,H) := \ker\,\p/\im\,\p
$$
are independent of the regular Hamiltonian perturbation
up to canonical isomorphism.  Theorem~\ref{thm:AC} then 
follows from the fact that $\HFHK_*(M,X,\tau_0;v,H)$ 
is isomorphic to $H_*(X;\Z_2)$. For the standard 
divergence-free frame on the $3$-sphere this was proved 
in~\cite[Section~5]{HNS2} and the argument 
again carries over verbatim to the general setting. 

\begin{remark}\label{rmk:floer1}\rm
Assume $v\in\sV^\reg$, let 
$
X=\H^n/\Lambda
$ 
be a hyperk\"ahler torus,
and let $\cF_\tau\subset\cF$ be a connected component of $\cF$.
Then the unperturbed Fueter equation $\dd_vf=0$ may 
have a nonconstant solution $f_0\in\cF_\tau$.  
(Examples are discussed in~\cite{HNS1,HNS2}.)
In this case~\eqref{eq:estimate} cannot hold for $f\in\cF_\tau$.  
However, $X$ is an additive group and, for $f\in\cF_\tau$, 
the difference $f-f_0$ lifts to a function with values 
in the universal cover $\H^n$.  Hence it follows from 
Lemma~\ref{le:estimate} that there is a constant 
$c_\tau>0$ such that every $f\in\cF_\tau$ 
satisfies the inequality
$$
\int_M\Abs{df}^2\,\dvol_M
\le c_\tau\int_M\left(\Abs{\dd_vf}^2+1\right)\,\dvol_M.
$$
Moreover, the restriction of the $1$-form $\Psi$ in~\eqref{eq:Psi}
to $\cF_\tau$ is still exact.  It is the differential of the action functional
$\sA_\tau:\cF_\tau\to\R$ given by
$$
\sA_\tau(f) := \frac12\int_M\inner{f-f_0}{\dd_vf}\,\dvol_M,\qquad
f\in\cF_\tau.
$$
To understand this, use the fact that the tangent 
bundle of $X$ is trivial, lift~${f-f_0}$ to a function with values 
in~$\H^n$, and note that the integral is independent
of the choice of the lift. Thus the construction 
of the Floer homology groups carries over to $\cF_\tau$. 
In favourable cases the Floer homo\-logy groups 
$\HFHK_*(M,X,\tau;v,H)$ can be computed with the 
methods of~\cite{PO}.  They are invariant under the action 
of the group of volume preserving diffeomorphism of $M$ 
on the set of triples $(\tau,v,H)$.  In general, 
they will not be invariant under deformation 
of the divergence-free frame~$v$.
\end{remark}

\begin{remark}\label{rmk:floer4}\rm
It would be interesting to understand the behavior of
the solutions of the Fueter equation~\eqref{eq:fueterH} 
as $v$ approaches a singular frame. 
\end{remark}

\begin{remark}\label{rmk:floer5}\rm
Another interesting question is whether the construction
of the Floer homology groups can be extended 
to nonflat target manifolds~$X$.  The key obstacle is 
noncompactness for the solutions to the Fueter equation. 
The expected phenomenon, which can be demonstrated in examples, 
is bubbling along codimension two submanifolds of $M$.
Important progress in understanding this phenomenon 
was recently made by Thomas Walpuski~\cite{Wa}.
His work will be an essential ingredient in the conjectural
development of a general hyperk\"ahler Floer theory.
\end{remark}

%%%%%%%%%%%%%%%%%%%%%%%%%%%%%%%%%%%%%%%%%%%%%% 
%%%%%%%%%%%%%%%%%%%%%%%%%%%%%%%%%%%%%%%%%%%%%% 
%%%%%%%%%%%%%%%% Section 5 %%%%%%%%%%%%%%%%%%% 
%%%%%%%%%%%%%%%%%%%%%%%%%%%%%%%%%%%%%%%%%%%%%% 
%%%%%%%%%%%%%%%%%%%%%%%%%%%%%%%%%%%%%%%%%%%%%% 
 
\section{Relation to Donaldson--Thomas theory}\label{sec:DT} 

The discussion in this section is speculative. 
It concerns the relation between the Donaldson--Thomas--Floer 
theory of a product manifold $Y=M\times\Sigma$
(where $\Sigma$ is a hyperk\"ahler $4$-manifold)
and hyperk\"ahler Floer theory.
 
Let $Y$ be a closed connected $7$-manifold.
A $3$-form $\phi\in\Om^3(Y)$ is called {\bf nondegenerate}
if, for any two linearly independent
tangent vectors $u,v$, there is a third tangent
vector $w$ such that $\phi(u,v,w)\ne 0$.
Every nondegenerate $3$-form $\phi$ determines a unique
Riemannian metric on $Y$ such that the bilinear form 
$(u,v)\mapsto u\times v$ on $TY$, 
defined by $\inner{u\times v}{w} := \phi(u,v,w)$,
is a cross product, i.e.\ it satisfies
$|u\times v|^2=|u|^2|v|^2-\langle u,v\rangle^2$
for all $u,v\in T_yY$ (see e.g.~\cite{SW}). 
A nondegenerate $3$-form also determines 
a unique orientation on $Y$ such that the $7$-form 
$\iota(u)\phi\wedge\iota(u)\phi\wedge\phi$ is 
positive for every nonzero tangent vector ${u\in TY}$.
A {\bf $\G_2$-structure} on $Y$ is a nondegenerate 
$3$-form $\phi$ that is harmonic with respect to
the Riemannian metric determined by $\phi$. 

Assume $\phi\in\Om^3(Y)$ is a $\G_2$-structure and
denote $\psi:=*\phi\in\Om^4(Y)$. Fix a compact Lie group $\G$
and let $\cA(Y)$ be the space of connections on a principal 
$\G$-bundle over~$Y$.  For this discussion it suffices to think 
of $\cA(Y)$ as the space of $1$-forms on $Y$ with 
values in the Lie algebra $\g=\Lie(\G)$. 
The {\bf Donaldson--Thomas--Chern--Simons functional}
$
\CS^\psi:\cA(Y)\to\R
$
is defined by
$$
\CS^\psi(\A) 
:= \int_Y\cs_{\A_0}(\A)\wedge\psi
$$
where $\A_0\in\cA(Y)$ is a reference connection and 
$\cs_{\A_0}(\A)\in\Om^3(Y)$ denotes the 
Chern--Simons $3$-form 
$$
\cs_{\A_0}(\A) 
:= \winner{a}{\left(F_{\A_0}
+\tfrac12d_{\A_0}a+\tfrac16[a\wedge a]\right)},\qquad
a:=\A-\A_0. 
$$
A critical point of $\CS^\psi$ is a connection $\A\in\cA(Y)$ 
whose curvature $F_\A$ satisfies
\begin{equation}\label{eq:DT1}
F_\A\wedge\psi = 0.
\end{equation}
The solutions are {\bf $\G_2$-instantons}.
A negative gradient flow line of $\CS^\psi$ is 
(gauge equivalent to) a pair, consisting of a path 
${\R\to\cA(Y):s\mapsto\A(s)}$ of connections and
a path ${\R\to\Om^0(Y,\g):s\mapsto\Phi(s)}$ of 
sections of the Lie algebra bundle, that satisfy 
\begin{equation}\label{eq:DT2}
\p_s\A-d_\A\Phi+*(F_\A\wedge\psi) = 0.
\end{equation}
The basic idea of Donaldson--Thomas theory in the $\G_2$ 
setting is to define Floer homology groups $\HFDT_*(Y)$, 
generated by the gauge equivalence classes of solutions 
of~\eqref{eq:DT1}, with the boundary operator given by 
counting the gauge equivalence classes of 
solutions of~\eqref{eq:DT2} (see~\cite{DT,DSe}).

Now let $M$ be a closed connected oriented $3$-manifold,
equipped with a volume form $\dvol_M$ and a normal 
regular divergence-free frame $v\in\sV^\reg$.
Denote by $\alpha_1,\alpha_2,\alpha_3$ the dual frame in $\Om^1(M)$. 
Then the $2$-forms $\alpha_i\wedge\alpha_j$ are closed
for all $i$ and~$j$ (see Remark~\ref{rmk:dualframe}).
Let $\Sigma$ be a closed connected hyperk\"ahler $4$-manifold
(i.e.\ either a $4$-torus or a $K3$ surface)
with symplectic forms $\sigma_1,\sigma_2,\sigma_3$ 
and complex structures $j_1,j_2,j_3$.
These structures determine a nondegenerate $3$-form
$\phi$ on the product
$
Y:=M\times\Sigma,
$ 
given by 
\begin{equation}\label{eq:phipsi}
\begin{split}
\phi &:= \dvol_M
- \alpha_1\wedge\sigma_1
- \alpha_2\wedge\sigma_2
- \alpha_3\wedge\sigma_3,\\
\psi &:= \dvol_\Sigma 
- \alpha_2\wedge\alpha_3\wedge\sigma_1
- \alpha_3\wedge\alpha_1\wedge\sigma_2
- \alpha_1\wedge\alpha_2\wedge\sigma_3.
\end{split}
\end{equation}
Here differential forms on $M$ and $\Sigma$ are identified
with their pullbacks to $Y=M\times\Sigma$. The Riemannian 
metric on $M\times\Sigma$ determined by $\phi$ is the product
metric and the cross product is given by
$v_1\times v_2=v_3$ and $\xi\times v_i=j_i\xi$
for $\xi\in T\Sigma$ and $i=1,2,3$. 

\bigbreak

Fix a principal $\G$-bundle $Q\to\Sigma$ and 
consider the product $\G$-bundle ${M\times Q\to M\times\Sigma}$.
Denote by $\cA(\Sigma)$ the space of connections on $Q$
and by $\g_Q\to\Sigma$ the Lie algebra bundle.
Write a connection on $\R\times M\times\Sigma$ 
in the form 
$
A + \Phi\,ds+\sum_i\Psi_i\alpha_i
$
with $A:\R\times M\to\cA(\Sigma)$ and 
$\Phi,\Psi_i:\R\times M\to\Om^0(\Sigma,\g_Q)$.
Then equation~\eqref{eq:DT2} has the following 
form (with $\eps=1$)
\begin{equation}\label{eq:DT3}
\begin{split}
\p_sA-d_A\Phi 
+ \sum_i\left(\p_{v_i}A-d_A\Psi_i\right)\circ j_i &= 0, \\
F_{01}+F_{23}-\eps^{-2}\inner{F_A}{\sigma_1} &= 0, \\
F_{02}+F_{31}-\eps^{-2}\inner{F_A}{\sigma_2} &= 0, \\
F_{03}+F_{12}-\eps^{-2}\inner{F_A}{\sigma_3} &= 0.
\end{split}
\end{equation}
Here $F_{0i}=F_{\A+\Phi\,ds}(\p/\p s,v_i)$ and
$F_{jk}:=F_{\A+\Phi\,ds}(v_j,v_k)$. Thus
\begin{equation}\label{eq:DT4}
\begin{split}
F_{0i} &:= \p_s\Psi-\p_{v_i}\Phi + [\Phi,\Psi_i], \\
F_{jk} &:= \p_{v_j}\Psi_k - \p_{v_k}\Psi_j + [\Psi_j,\Psi_k] 
+ \sum_i\alpha_i([v_j,v_k])\Psi_i.
\end{split}
\end{equation}
For general $\eps$ equation~\eqref{eq:DT3} is obtained
by replacing $\sigma_i$ with $\eps^2\sigma_i$
and $\dvol_\Sigma$ with $\eps^4\dvol_\Sigma$.
Taking the limit $\eps\to0$ in~\eqref{eq:DT3}
one obtains the equation
\begin{equation}\label{eq:DT5}
\p_sA-d_A\Phi 
- \sum_i\left(\p_{v_i}A-d_A\Psi_i\right)\circ j_i 
= 0,\qquad 
F_A^+=0.
\end{equation}
This is the unperturbed Fueter equation~\eqref{eq:floerH} 
on ${\R\times M}$ with values in the moduli space
$X=\cM^{\mathrm{asd}}(\Sigma)$ of anti-self-dual
instantons on $\Sigma$ with its standard hyperk\"ahler
structure (see Remark~\ref{rmk:Gfueter5} below).
These observations suggests a correspondence 
between the Donaldson--Thomas--Floer homology groups 
$\HFDT_*(M\times\Sigma)$ and the 
hyperk\"ahler Floer homology groups 
$\HFHK_*(M,\cM^{\mathrm{asd}}(\Sigma))$,
in analogy to the Atiyah--Floer conjecture~\cite{DS}.

\begin{remark}\label{rmk:DT1}\rm
{\bf (i)}  The $4$-form $\psi$ in~\eqref{eq:phipsi}
is closed when $\lambda=\dvol_M(v_1,v_2,v_3)$ is constant,
however, the $3$-form $\phi$ will in general not be closed.

\smallskip\noindent{\bf (ii)}
Consider the $7$-manifold 
$
Y:=M\times\Sigma,
$
where $\Sigma$ is a $K3$-surface and $M$ is a 
closed oriented $3$-manifold with ${b_1(M)\le 2}$. 
Let $\pi:Y\to\Sigma$ be the projection. 
The following argument by Donaldson shows that 
$Y$ does not admit a $\G_2$-structure. 
On a $\G_2$-manifold there is a splitting of the space 
$H^2(Y)$ of harmonic $2$-forms into the eigenspaces 
\begin{equation*}
\begin{split}
H^{2,+}(Y)
&:=\left\{\tau\in H^2(Y)\,|\,*(\phi\wedge\tau)=2\tau\right\}
= \left\{*(\psi\wedge\alpha)\,|\,\alpha\in H^1(Y)\right\}, \\
H^{2,-}(Y)
&:=\left\{\tau\in H^2(Y)\,|\,*(\phi\wedge\tau)=-\tau\right\}
= \left\{\tau\in H^2(Y)\,|\,\psi\wedge\tau=0\right\}.
\end{split}
\end{equation*}
The subspace $H^{2,+}(Y)$ is isomorphic to $H^1(Y)$ 
and, by definition, the qua\-dra\-tic form 
$H^{2,-}(Y)\to\R:\tau\mapsto \int_Y\phi\wedge\tau\wedge\tau$
is negative definite.  In the case at hand, $H^{2,-}(Y)$ 
has codimension $b_1(Y)=b_1(M)\le2$, and hence 
intersects $\pi^*H^{2,\pm}(\Sigma)$ nontrivially.  
Choose $0\ne\tau^\pm\in H^{2,\pm}(\Sigma)$ such that 
$\pi^*\tau^\pm\in H^{2,-}(Y)$ and 
$\tau^\pm\wedge\tau^\pm=\pm\dvol_\Sigma$.
Then $\int_Y\phi\wedge\pi^*(\tau^-\wedge\tau^-)$
and $\int_Y\phi\wedge\pi^*(\tau^+\wedge\tau^+)$
have opposite signs, a contradiction. 

\smallskip\noindent{\bf (iii)}
The argument in~(ii) breaks down for $M=\T^3$ 
and in this case the $3$-form $\phi$ in~\eqref{eq:phipsi} 
is indeed a $\G_2$-structure (for a suitable frame on $\T^3$).

\smallskip\noindent{\bf (iv)}
Existence results for $\G_2$-structures were established by Joyce~\cite{J}.
A question posed by Donaldson is, which $7$-manifolds 
admit nondegenerate $3$-forms $\phi$ that are closed
or co-closed, but not necessarily both. This is analogous 
to the question, which manifolds admit symplectic or 
complex structures, but not necessarily K\"ahler structures.
\end{remark}

\begin{remark}\label{rmk:DT2}\rm
The above setting extends naturally to general 
$\G$-bundles over $M\times\Sigma$.  
A homotopy class of maps from $M$ to the moduli space 
$\cM^{\mathrm{asd}}(\Sigma)$ determines the relevant 
principal $\G$-bundle over $M\times\Sigma$.
\end{remark}

\begin{remark}\label{rmk:DT3}\rm
The discussion of the present section is closely related 
to several observations by Donaldson--Thomas in~\cite{DT} and 
by Donaldson--Segal in~\cite{DSe}. In~\cite[Section~5]{DT}
Donaldson and Thomas discuss $\Spin(7)$-instantons on a product
$S\times\Sigma$ of two hyperk\"ahler $4$-manifolds $S$ and $\Sigma$.
They note that, shrinking the metric on $\Sigma$, leads in the 
adiabatic limit to solutions of the Fueter equation 
for maps $u:S\to\cM^{\mathrm{asd}}(\Sigma)$.
Taking $S=\R\times M$ one arrives at equation~\eqref{eq:DT3}.

In~\cite[Section~6]{DSe} Donaldson and Segal extend this 
discussion to a setting where $M\times\Sigma$ is replaced
by a $\G_2$-manifold $Y$, and $M$ is replaced by an 
associative submanifold of $Y$.  This leads to an interaction 
between $\G_2$-instantons on $Y$ and Fueter sections 
of the bundle of framed anti-self-dual instantons 
on the fibers of the normal bundle of $M$. Generically, 
such Fueter sections are expected to exist at 
isolated parameters in a $1$-parameter family 
of $\G_2$ structures. The existence of nonconstant solutions
of the Fueter equation for singular divergence-free 
frames, as discussed in Section~\ref{sec:FUETER}, 
seems to be a linear analogue of this codimension 
one phenomenon.

In this extended setting, relating the Donaldson--Thomas 
equations over a $G_2$-manifold $Y$ to the Fueter 
equations over an associative submanifold~$M$, 
important progress has recently been made by 
Walpuski~\cite{Wa1,Wa2}.  He carried out the 
adiabatic limit analysis and proved that Fueter 
sections over $M$, under suitable transversality 
assumptions, give rise to $\G_2$-instantons over $Y$
whose energy is concentrated near $M$.
\end{remark}

%%%%%%%%%%%%%%%%%%%%%%%%%%%%%%%%%%%%%%%%%%%%%% 
%%%%%%%%%%%%%%%%%%%%%%%%%%%%%%%%%%%%%%%%%%%%%% 
%%%%%%%%%%%%%%%% Section 6 %%%%%%%%%%%%%%%%%%% 
%%%%%%%%%%%%%%%%%%%%%%%%%%%%%%%%%%%%%%%%%%%%%% 
%%%%%%%%%%%%%%%%%%%%%%%%%%%%%%%%%%%%%%%%%%%%%% 
 
\section{The gauged Fueter equation}\label{sec:GAUGE} 

Equation~\eqref{eq:DT3} extends naturally to a setting 
where the space of connections over $\Sigma$ is replaced 
by a hyperk\"ahler manifold $(X,\om_1,\om_2,\om_3,J_1,J_2,J_3)$
and the group of gauge transformations over $\Sigma$ 
by a compact Lie group~$\G$, acting on~$X$ by hyperk\"ahler isometries.  
Denote the action by $(g,x)\mapsto gx$ 
and the infinitesimal action of the Lie algebra
$\g:=\Lie(\G)$ by $L_x:\g\to T_xX$.
Thus $L_x\xi:=\left.\frac{d}{dt}\right|_{t=0}\exp(t\xi)x$
for $x\in X$ and $\xi\in\g$. Choose an invariant inner product on $\g$
and suppose that the action is generated by equivariant moment maps 
$\mu_1,\mu_2,\mu_3:X\to\g$, so that 
$$
\om_i(L_x\xi,\hat x) = \inner{d\mu_i(x)\hat x}{\xi},\qquad
\om_i(L_x\xi,L_x\eta) = \inner{\mu_i(x)}{[\xi,\eta]}, 
$$
for all $\hat x\in T_xX$, $\xi,\eta\in\g$, and $i=1,2,3$. 
Fix an oriented $3$-manifold $M$ with a volume form 
$\dvol_M$, a normal divergence-free frame $v_1,v_2,v_3$,
and denote by $\alpha_1,\alpha_2,\alpha_3\in\Om^1(M)$ 
the dual frame.  Choose a principal $\G$-bundle 
${\pi:P\to M}$, let $\cA\subset\Om^1(P,\g)$ be 
the space of connections on $P$, and let $\cF$
be the space of $\G$-equivariant maps $f:P\to X$. 
There is a natural $1$-form on $\cA\times\cF$, 
which assigns to every pair $(A,f)\in\cA\times\cF$ 
the linear map 
$
\Psi_{A,f}:T_A\cA\times T_f\cF\to\R,
$
given by
\begin{equation}\label{eq:Gaction}
\begin{split}
\Psi_{A,f}\left(\hat A,\hat f\right) 
&:= 
\int_M\bigl\langle F_A\wedge \hat A\bigr\rangle \\
&\quad\;\; - \sum_i\int_M\Bigl(
\om_i\bigl(d_Af(v_i),\hat f\bigr)
+ \bigl\langle\mu_i(f),\hat A(v_i)\bigr\rangle
\Bigr)\,\dvol_M \\
&\;=
\sum_i\int_M
\Bigl\langle
F_A(v_j,v_k)-\mu_i(f),\hat A(v_i)
\Bigr\rangle\dvol_M \\
&\quad\;\;
- \int_M
\Bigl\langle
\sum_iJ_id_Af(v_i),\hat f
\Bigr\rangle\,\dvol_M 
\end{split}
\end{equation}
for $\hat A\in T_A\cA=\Om^1(M,\g_P)$ 
and $\hat f\in T_f\cF=\Om^0(M,f^*TX/G)$.
Here the second sum runs over all cyclic 
permutations $i,j,k$ of $1,2,3$. 
The $1$-form $d_Af:TP\to f^*TX$ 
is the covariant derivative of $f$ 
with respect to $A$, defined by 
$(d_Af)_p(\hat p) := df(p)\hat p+L_{f(p)}A_p(\hat p)$
for $\hat p\in T_pP$. It is $\G$-equivariant
and horizontal (i.e.\ $(d_Af)_p(p\xi)=0$ for $\xi\in\g$).  
Hence it descends to a $1$-form on~$M$ with values 
in the quotient bundle ${f^*TX/\G\to M}$.
To understand the term $d_Af(v_i)\in\Om^0(M,f^*TX/\G)$,
choose $\G$-equivariant lifts ${\tilde v_i\in\Vect(P)}$ 
of~$v_i$ and observe that the section $d_Af(\tilde v_i)$ 
of the vector bundle ${f^*TX\to P}$ is $\G$-equivariant 
and independent of the choice of the lifts.
The resulting section of the bundle $f^*TX/\G\to P/\G=M$ 
is denoted $d_Af(v_i)$. 

The group $\cG=\cG(P)$ of gauge transformations 
acts contravariantly on $\cA\times\cF$ by 
$g^*A := g^{-1}dg+g^{-1}Ag$ and $g^*f:=g^{-1}f$.
The covariant infinitesimal action of the Lie 
algebra $\Om^0(M,\g_P)=\Lie(\cG)$ is given by 
$\Phi\mapsto(-d_A\Phi,L_f\Phi)\in T_A\cA\times T_f\cF$.
The $1$-form~\eqref{eq:Gaction} is $\cG$-invariant 
and horizontal, in the sense that 
$\Psi_{A,f}(-d_A\Phi,L_f\Phi)=0$
for all $A$, $f$, and $\Phi$. 
Hence $\Psi$ descends to a $1$-form on 
the quotient space $\cB:=(\cA\times\cF)/\cG$. 

\begin{remark}\label{rmk:Gfueter1}\rm
The $1$-form~\eqref{eq:Gaction} is closed.
To see this, choose a smooth path 
${I\to\cA\times\cF:s\mapsto(A_s,u_s)}$.
Think of $\A:=\{A_s\}_{s\in I}$ as a connection 
on the principal bundle $I\times P$ over $I\times M$,
and of $u$ as a $\G$-equivariant map from
$I\times P$ to $X$. The integral of $\Psi$ 
over this path is given by
\begin{equation*}
\begin{split}
\int_I(\A,u)^*\Psi
&=
\int_I\Psi_{A_s,f_s}\left(\p_sA,\p_su\right)\,ds \\
&= \int_{I\times M} \Bigl(
\frac12\winner{F_\A}{F_\A} 
- \sum_i\bigl(u^*\om_i-d\inner{\mu_i(u)}{\A}\bigr)
\wedge \iota(v_i)\dvol_M\Bigr).
\end{split}
\end{equation*}
The last integral is meaningful, because the $2$-form
$u^*\om_i-d\inner{\mu_i(u)}{\A}$ on ${I\times P}$ 
descends to $I\times M$. Since the integrand is closed, 
the integral is invariant under homotopy with fixed endpoints.  
If $\iota(v_i)\dvol_M=d\beta_i$ and $\CS:\cA\to\R$ 
denotes the Chern--Simons functional, 
then~\eqref{eq:Gaction} is the differential 
of the action functional
\begin{equation}\label{eq:GACTION}
\sA(A,f) 
:= \CS(A) + \sum_i\int_M\beta_i\wedge
\left(f^*\om_i-d\inner{\mu_i(f)}{A}\right).
\end{equation}
\end{remark}

\begin{remark}\label{rmk:Gfueter2}\rm
Given $A\in\cA$, define the twisted Fueter operator by
$$
\dd_{A,v}f := J_1d_Af(v_1)+J_2d_Af(v_2)+J_3d_Af(v_3)
$$
for $f\in\cF$.  Thus $\dd_{A,v}f$ is a section of the 
quotient bundle $f^*TX/\G\to M$. 
Then the zeros of the $1$-form~\eqref{eq:Gaction}
are the solutions $(A,f)\in\cA\times\cF$
of the {\bf three-dimensional 
gauged Fueter equation}
\begin{equation}\label{eq:Gfueter3}
\dd_{A,v}f =0,\qquad
*F_A = \sum_i(\mu_i\circ f)\pi^*\alpha_i.
\end{equation}
Here $*$ denotes the Hodge $*$-operator on $M$ 
associated to the metric~\eqref{eq:metric}.
Thus 
$$
*F_A=\sum_iF_A(v_j,v_k)\pi^*\alpha_i, 
$$
where the sum runs over all cyclic permutations 
$i,j,k$ of $1,2,3$. 

The gradient flow lines are pairs, 
consisting of a connection $\A=A+\Phi\,ds$ on $\R\times P$
and a $\G$-equivariant map $u:\R\times P\to X$, that satisfy
the {\bf four-dimensional gauged Fueter equation}
\begin{equation}\label{eq:Gfueter4}
\p_su+L_u\Phi-\dd_{A,v}u = 0,\qquad
\p_sA-d_A\Phi + *F_A = \sum_i(\mu_i\circ u)\pi^*\alpha_i.
\end{equation}
This is reminiscent of the symplectic vortex 
equations~\cite{CGS,CGMS,F,GS,M2,O,OZ}.  
Similar equations were studied by Taubes~\cite{T},
Pidstrigatch~\cite{P}, and Haydys~\cite{H1}.
The usual Fueter equation corresponds to the case $\G=\{1\}$,
the instanton Floer equation~\cite{F4} to the case $X=\{\mathrm{pt}\}$,
the Donaldson--Thomas equation to the case 
$X=\cA(\Sigma)$ and $\G=\cG(\Sigma)$ 
(see Remark~\ref{rmk:Gfueter5}), and
the Seiberg--Witten equation to the case 
$X=\H$ and $\G=S^1$ (see Section~\ref{sec:SW}).
\end{remark}

\begin{remark}\label{rmk:Gfueter3}\rm
Define the energy of a pair $(A,f)\in\cA\times\cF$ by 
\begin{equation}\label{eq:Genergy}
\cE(A,f) := \frac12\int_M\Bigl(
\Abs{d_Af}^2 +\Abs{F_A}^2+\sum_i\Abs{\mu_i(f)}^2
\Bigr)\,\dvol_M.
\end{equation}
Then there is an {\bf energy identity}
\begin{equation}\label{eq:GENERGY}
\begin{split}
&\frac12\int_M\Bigl(\bigl|\dd_{A,v}f\bigr|^2
+\bigl|*F_A-\sum_i\mu_i(f)\pi^*\alpha_i\bigr|^2
\Bigr)\,\dvol_M \\
&\qquad= \cE(A,f) + \sum_i\int_M\alpha_i\wedge
\Bigl(f^*\om_i-d\inner{\mu_i(f)}{A}\Bigr).
\end{split}
\end{equation}
This is the gauged analogue of equation~\eqref{eq:ENERGY}.
\end{remark}

\begin{remark}\label{rmk:Gfueter4}\rm
One can introduce an $\eps$-parameter 
in~\eqref{eq:Gfueter4} as follows
\begin{equation}\label{eq:Gfueter4eps}
\p_su+L_u\Phi-\dd_{A,v}u = 0,\qquad
\p_sA-d_A\Phi + *F_A 
= \eps^{-2}\sum_i(\mu_i\circ u)\pi^*\alpha_i.
\end{equation}
In the limit $\eps\to0$ one obtains, formally, 
the equation
\begin{equation}\label{eq:Gfueter0}
\p_su+L_u\Phi-\dd_{A,v}u = 0,\quad
\mu_1\circ u = \mu_2\circ u = \mu_3\circ u=0.
\end{equation}
This corresponds to the four-dimensional 
Fueter equation on $\R\times M$ with values
in the hyperk\"ahler quotient 
$$
X\dslash\G :=
(\mu_1^{-1}(0)\cap\mu_2^{-1}(0)\cap\mu_3^{-1}(0))/\G.
$$
\end{remark}

\begin{remark}\label{rmk:Gfueter5}\rm
The space $\cX:=\cA(\Sigma)$ 
of connections on a principal $\G$-bundle $Q$ 
over a closed hyperk\"ahler $4$-manifold 
$(\Sigma,\sigma_1,\sigma_2,\sigma_3,j_1,j_2,j_3)$
is itself a hyperk\"ahler manifold.
The symplectic forms $\om_i\in\Om^2(\cX)$ and the 
complex structures $J_i:T\cX\to T\cX$ are given by
$$
\om_i(\alpha,\beta):=
\int_\Sigma\winner{\alpha}{\beta}\wedge\sigma_i,\qquad
J_i\alpha 
:= *_\Sigma(\alpha\wedge\sigma_i) 
= -\alpha\circ j_i
$$
for $\alpha,\beta\in\Om^1(\Sigma,\g_Q)=T_A\cA(\Sigma)$
and $i=1,2,3$. The group $\cG=\cG(\Sigma)$ of gauge 
transformations of $Q$ acts on $\cX=\cA(\Sigma)$ by 
hyperk\"ahler isometries and the moment maps 
$\mu_i:\cA(\Sigma)\to\Om^0(\Sigma,\g_Q)=\Lie(\cG(\Sigma))$
are 
$$
\mu_i(A) := *_\Sigma(F_A\wedge\sigma_i) = \inner{F_A}{\sigma_i},\qquad
A\in\cA(\Sigma).
$$ 
If $\cP=M\times\cG(\Sigma)$ is the trivial bundle 
with the infinite-dimensional structure group $\cG(\Sigma)$
then~\eqref{eq:Gfueter4eps} is equivalent to the 
Donaldson--Thomas equation~\eqref{eq:DT3} 
on $M\times\Sigma$. The function 
$A:\R\times M\to\cA(\Sigma)$ in~\eqref{eq:DT3} 
corresponds to $u$ in~\eqref{eq:Gfueter4eps},
while the functions $\Psi_i:\R\times M\to\Om^0(\Sigma,\g_Q)$ 
in~\eqref{eq:DT3} determine the path of connections 
$A(s):=\sum_i\Psi_i(s)\alpha_i\in\Om^1(M,\Om^0(\Sigma,\g_Q))$ 
in~\eqref{eq:Gfueter4eps}.  
The hyperk\"ahler quotient is the moduli space 
$\cA(\Sigma)\dslash\cG(\Sigma)=\cM^{\mathrm{asd}}(\Sigma)$
of anti-self-dual instantons on $\Sigma$.
These observations extend to nontrivial 
$\cG(\Sigma)$-bundles $\cP\to M$, arising from 
general $\G$-bundles over $M\times\Sigma$.
\end{remark}

%%%%%%%%%%%%%%%%%%%%%%%%%%%%%%%%%%%%%%%%%%%%%% 
%%%%%%%%%%%%%%%%%%%%%%%%%%%%%%%%%%%%%%%%%%%%%% 
%%%%%%%%%%%%%%%% Section 7 %%%%%%%%%%%%%%%%%%% 
%%%%%%%%%%%%%%%%%%%%%%%%%%%%%%%%%%%%%%%%%%%%%% 
%%%%%%%%%%%%%%%%%%%%%%%%%%%%%%%%%%%%%%%%%%%%%% 
 
\section{Relation to Seiberg--Witten theory}\label{sec:SW} 

Let $M$ be a closed connected oriented Riemannian $3$-manifold.
A spin$^c$ structure on $M$ is a rank two complex hermitian 
vector bundle $W\to M$ equipped with a Clifford action 
$\gamma:TM\to\End(W)$. This action assigns to every 
tangent vector $v\in T_yM$ a traceless endomorphism 
$\gamma(v)$ of the fiber $W_y$ satisfying
$$
\gamma(v)+\gamma(v)^* = 0,\qquad \gamma(v)^*\gamma(v) = \abs{v}^2\one,
$$
and it satisfies $\gamma(v_3)\gamma(v_2)\gamma(v_1)=\one$
for every positive orthonormal frame. 
A spin$^c$ connection is a hermitian connection on $W$
that satisfies the Leibniz rule for Clifford multiplication 
(with the Levi-Civita connection on the tangent bundle).
Associated to a spin$^c$ connection $\Nabla{A}$ 
is a Dirac operator 
$
\dD_A:\Om^0(M,W)\to\Om^0(M,W)
$
defined by 
$$
\dD_Af := \sum_i\gamma(v_i)\Nabla{A,v_i}f
$$
for $f\in\Om^0(M,W)$. Here $v_1,v_2,v_3$ is any global
orthonormal frame of $TM$.  The Dirac operator is 
self-adjoint and independent of the choice of the frame.
The difference of two spin$^c$ connections is an 
imaginary valued $1$-form.  Let $\cA(\gamma)$ denote the 
space of spin$^c$ connections on $W$. The perturbed 
{\bf Chern--Simons--Dirac functional} 
$\CSD_\eta:\cA(\gamma)\times\Om^0(M,W)$ has the form
$$
\CSD_\eta(A,f) 
:= 
\int_Y\winner{(A-A_0)}{(\eta+\tfrac12(F_{A_0}+F_A))}
- \frac12\int_Y\inner{f}{\dD_Af}\,\dvol_M.
$$
Here $A_0\in\cA(\gamma)$ is a reference connection,
$\eta\in\Om^2(M,\i\R)$ is a closed $2$-form,
and $F_A:=\tfrac12\trace(F^{\Nabla{A}})\in\Om^2(M,\i\R)$.
A negative gradient flow line of $\CSD_\eta$ is a triple
$(A,\Phi,u)$, consisting of a smooth path
${\R\to\cA(\gamma):s\mapsto A(s)}$
of spin$^c$ connections, a smooth path
${\R\to\Om^0(M,\i\R):s\mapsto\Phi(s)}$  
of functions on~$M$, and a smooth path
${\R\to\Om^0(M,W):s\mapsto u_s=u(s,\cdot)}$ 
of sections of~$W$, that satisfy the $4$-dimensional 
{\bf Seiberg--Witten--Floer equation}
\begin{equation}\label{eq:SWF}
\p_su+\Phi u - \dD_Au=0,\qquad
\p_sA-d\Phi+*(F_A+\eta)=\gamma^{-1}((uu^*)_0).
\end{equation}
Here $(uu^*)_0:W\to W$ denotes the traceless hermitian 
endomorphism given by $(uu^*)_0w:=\inner{u}{w}u-\tfrac12\abs{u}^2w$,
$I$ denotes the complex structure on $W$, $\inner{\cdot}{\cdot}$
denotes the real inner product on $W$, and 
$\gamma^{-1}((uu^*)_0) := \tfrac{\i}{2}\inner{I\gamma(\cdot)u}{u}$.
(See the book by Kronheimer--Mrowka~\cite{KM} 
for a detailed account of Seiberg--Witten--Floer theory
and its applications to low-dimensional topology.)

Now let $v_1,v_2,v_3\in\Vect(M)$ be an orthonormal divergence-free 
frame and let $\alpha_1,\alpha_2,\alpha_3\in\Om^1(M)$ be the dual frame.  
For an orthonormal frame the divergence-free condition 
can be expressed in the form 
\begin{equation}\label{eq:ONBdiv}
\Nabla{v_1}v_1 + \Nabla{v_2}v_2 + \Nabla{v_3}v_3 = 0.
\end{equation}
The frame induces a spin structure on the trivial bundle $W:=M\times\H$
via
$$
\gamma(v_i) = J_i,\qquad i=1,2,3,
$$
where $J_1,J_2,J_3$ are the complex structures in~\eqref{eq:IJ}.
(This is a spin structure because it commutes with all three 
complex structures $I_1,I_2,I_3$ in~\eqref{eq:IJ}.)  
The spin connection $A_0\in\cA(\gamma)$ of this structure 
is given by 
$$
\Nabla{A_0,v_i}f = \p_{v_i}f + A_0(v_i)f
$$
for $f\in\Om^0(M,\H)$ and $i=1,2,3$, where
$$
A_0(v_i) := \frac12\inner{\Nabla{v_i}v_j}{v_k}J_i 
+ \frac12\inner{\Nabla{v_i}v_i}{v_j}J_k
- \frac12\inner{\Nabla{v_i}v_i}{v_k}J_j
$$
for every cyclic permutation $i,j,k$ of $1,2,3$. 
The curvature of $A_0$ is traceless.
A simple calculation, using~\eqref{eq:ONBdiv} and the
orthonormal condition, shows that 
\begin{equation}\label{eq:dirac0}
\dD_{A_0}f = \dd_vf + \lambda f,\qquad 
\lambda:=\frac14\sum_i\alpha_i([v_j,v_k]),
\end{equation}
where the sum runs over all cyclic permutations $i,j,k$ of $1,2,3$.
Now consider the circle action on $\H$ generated by the vector field
$x\mapsto -x\i$.  This is the standard circle action associated to the 
complex structure $I_1$ in~\eqref{eq:IJ}.  It preserves the
hyperk\"ahler structure determined by the complex structures 
$J_1,J_2,J_3$ and the symplectic forms $\om_1,\om_2,\om_3$
in~\eqref{eq:ENERGY}.  The moment maps $\mu_1,\mu_2,\mu_3:\H\to\i\R$ 
of this action are given by $\mu_i(x)=\frac{\i}{2}\om_i(-x\i,x)$.  
Hence
\begin{equation}\label{eq:gammu}
\gamma^{-1}((uu^*)_0) = \frac{\i}{2}\inner{-\gamma(\cdot)u\i}{u}
= \sum_i\frac{\i}{2}\om_i(-u\i,u)\alpha_i
= \sum_i(\mu_i\circ u)\alpha_i.
\end{equation}
By~\eqref{eq:dirac0} and~\eqref{eq:gammu}, 
the Seiberg--Witten--Floer equation~\eqref{eq:SWF} 
has the form
\begin{equation}\label{eq:SWFdiv}
\p_su-u\Phi - \dd_{A,v}u=\lambda u,\quad
\p_sA-d\Phi+*\left(dA+\eta\right)
=\sum_i(\mu_i\circ u)\alpha_i.
\end{equation}
Here $\R\to\Om^1(M,\i\R):s\mapsto A(s)$ is a path of imaginary valued
$1$-forms and the associated path of spin$^c$ connections is 
$s\mapsto A_0+A(s)$.  With $\eta=0$ and $\lambda=0$ 
this is the gauged Fueter equation~\eqref{eq:Gfueter4} 
for~${X=\H}$ and~${\G=S^1}$.  This correspondence extends to general 
spin$^c$ structures via the appropriate circle bundles over $M$. 
Replacing the circle by the group $\G=\Sp(1)$, acting on $\H$ on the right,
one obtains the equations of Pidstrigatch--Tyurin~\cite{PT}.

%%%%%%%%%%%%%%%%%%%%%%%%%%%%%%%%%%%%%%%%%%%%%% 
%%%%%%%%%%%%%%%%%%%%%%%%%%%%%%%%%%%%%%%%%%%%%% 
%%%%%%%%%%%%%%%% Appendix A %%%%%%%%%%%%%%%%%% 
%%%%%%%%%%%%%%%%%%%%%%%%%%%%%%%%%%%%%%%%%%%%%% 
%%%%%%%%%%%%%%%%%%%%%%%%%%%%%%%%%%%%%%%%%%%%%% 

\appendix

\section{Divergence-free frames} \label{app:DFF}  

Let $M$ be a closed connected oriented $3$-manifold and $\dvol_M\in\Om^3(M)$ 
be a positive volume form.  Denote the set of positive frames by 
$$
\sF := \left\{v=(v_1,v_2,v_3)\in\Vect(M)^3\,|\,
\dvol_M(v_1,v_2,v_3)>0\right\}
$$
and the set of divergence-free positive frames by 
$$
\sV := \left\{v\in\sF\,|\,
d\iota(v_i)\dvol_M=0\mbox{ for }i=1,2,3\right\}.
$$
Given three deRham cohomology classes $a_1,a_2,a_3\in H^2(M;\R)$, 
denote the set of divergence-free positive frames that 
represent the classes $a_i$ by
$$
\sV_a := \left\{v\in\sV\,|\,
[\iota(v_i)\dvol_M]=a_i\mbox{ for }i=1,2,3\right\}.
$$
The following theorem is an application of Gromov's h-principle 
and is proved in~\cite[page~182]{G} and~\cite[Corollary~20.4.3]{EM}.
Although the result is stated in these references in the weaker form
that every frame can be deformed in $\sF$ to a divergence-free frame,
the proofs give the stronger result stated below 
(as explained to the author by Yasha Eliashberg).

\begin{theorem}[{\bf Gromov}]\label{thm:gromov}
The inclusion $\sV_a\hookrightarrow\sF$ is a homotopy equivalence
for all $a_1,a_2,a_3\in H^2(M;\R)$,
and so is the inclusion $\sV\hookrightarrow\sF$.
\end{theorem}

\begin{lemma}[{\bf Gromov}]\label{le:ample}
Let $V$ be a $3$-dimensional real vector space and
$S:V\times V\to V$ be a skew-symmetric bilinear map.
Let $\cR\subset\Hom(V,\End(V))$ be the set of 
all linear maps $L:V\to\End(V)$ such that 
the map
$$
\Lambda^2V\to V:u\wedge v\mapsto
S(u,v)+L(u)v-L(v)u
$$
is a vector space isomorphism.
Fix a $2$-dimensional linear
subspace $E\subset V$ and a linear map
$\lambda:E\to\End(V)$.  Define 
$$
\cL := \left\{L\in\Hom(V,\End(V))\,|\,L|_E=\lambda\right\}.
$$
If $\cL\cap\cR$ is nonempty, then it has two connected 
components and the convex hull of each connected 
component of $\cL\cap\cR$ is equal to $\cL$.
\end{lemma}

%\begin{lemma}[{\bf Gromov}]\label{le:ample}
%Let $S:\R^3\times\R^3\to\R^3$ be a skew-symmetric bilinear map.
%Let $\cR\subset\Hom(\R^3,\End(\R^3))$ be the set of 
%all linear maps $L:\R^3\to\End(\R^3)$ that satisfy the 
%nondegeneracy condition
%\begin{equation}\label{eq:R}
%\inner{w}{S(u,v)+L(u)v-L(v)u}=0\quad\forall\;u,v\in\R^3
%\qquad\implies\qquad w=0
%\end{equation}
%for every $w\in\R^3$. Fix a $2$-dimensional linear
%subspace $E\subset\R^3$ and a linear map
%$\lambda:E\to\End(\R^3)$.  Define 
%$$
%\cL := \left\{L\in\Hom(\R^3,\End(\R^3))\,|\,L|_E=\lambda\right\}.
%$$
%If $\cL\cap\cR$ is nonempty, then it has two connected 
%components and the convex hull of each connected 
%component of $\cL\cap\cR$ is equal to $\cL$.
%\end{lemma}

\begin{proof}
The proof is a special case of the argument given by 
Eliashberg--Mishachev in~\cite[pages 183/184]{EM}.
Assume without loss of generality that
$$
V=\R^3,\qquad E=\left\{x\in\R^3\,|\,x_1=0\right\}.
$$
Write $S$ in the form
$$
S(u,v) =: \sum_{i<j}u_iv_jS_{ij},\qquad
S_{ij}=-S_{ji}\in\R^3,
$$
and write a linear map $L:\R^3\to\End(\R^3)$ in the form 
$$
L(u)v = \sum_{i,j=1}^3u_iv_jL_{ij},\qquad L_{ij}\in\R^3.
$$
Then $L\in\cR$ if and only if
\begin{equation}\label{eq:det}
\det(S_{23}+L_{23}-L_{32},S_{31}+L_{31}-L_{13},S_{12}+L_{12}-L_{21})\ne 0.
\end{equation}
Denote by $\cR^+$, respectively $\cR^-$, the set of all 
$L\in\cR$ for which the sign of the determinant 
in~\eqref{eq:det} is positive, respectively negative.

Fix a linear map $\lambda:E\to\End(\R^3)$.  This map 
is determined by the coefficients $L_{ij}$ with $i=2,3$.  
Thus an element $L\in\cL\cap\cR$ is determined by 
the choice of $L_{11}$, $L_{12}$, $L_{13}$. 
If $S_{23}+L_{23}-L_{32}=0$ then the determinant 
in~\eqref{eq:det} vanishes for every $L\in\cL$ 
and so $\cL\cap\cR=\emptyset$. 
Hence assume $S_{23}+L_{23}-L_{32}\ne0$. 
Then $\cL\cap\cR^+$ and $\cL\cap\cR^-$ are 
nonempty connected submanifolds of $\R^3\times\R^3\times\R^3$.
(Namely, $L_{11}$ is any vector in $\R^3$, 
$L_{12}$ is required to be in the complement 
of an affine line, and then $L_{13}$ is required 
to be in the complement of an affine plane 
depending smoothly on $L_{12}$.)

Choose $x,y\in\R^3$ such that 
$$
\det(S_{23}+L_{23}-L_{32},x,y) > 0.
$$
Then, for $t>0$ sufficiently large, 
\begin{equation*}
\begin{split}
&\det\left(S_{23}+L_{23}-L_{32},S_{31}+tx,S_{12}+ty\right) > 0,\\
&\det\left(S_{23}+L_{23}-L_{32},S_{31}-tx,S_{12}-ty\right) > 0.
\end{split}
\end{equation*}
Given $L\in\cL$ choose $L',L''\in\cL$ such that 
\begin{equation*}
\begin{split}
&
L_{11}'=L_{11},\qquad 
L_{12}':=L_{21}+ty,\qquad
L_{13}':=L_{31}-tx,\\
&
L_{11}''=L_{11},\qquad 
L_{12}'':=L_{21}-ty,\qquad
L_{13}'':=L_{31}+tx.
\end{split}
\end{equation*}
Then $L',L''\in\cL\cap\cR^+$ and $L=\tfrac12(L'+L'')$.
Hence the convex hull of $\cL\cap\cR^+$ is equal to $\cL$.
A similar argument shows that the convex hull of $\cL\cap\cR^-$ 
is also equal to $\cL$.  This proves Lemma~\ref{le:ample}.
\end{proof}

\begin{proof}[Proof of Theorem~\ref{thm:gromov}]
Fix a Riemannian metric on $M$, let $a\in H^2(M;\R)^3$, 
and let $\sigma_i\in\Om^2(M)$ be the harmonic 
representative of $a_i$, $i=1,2,3$. Define 
$$
\sH_a:=\left\{\beta=(\beta_1,\beta_2,\beta_3)\in\Om^1(M)^3\,|\,
\sigma_i+d\beta_i\mbox{ are linearly independent}\right\}.
$$
Let 
$$
\pi_a:\sH_a\to\sV_a
$$ 
be the projection defined by 
$\pi_a(\beta)=v$ for $\beta\in\sH_a$, 
where $\iota(v_i)\dvol_M:=\sigma_i+d\beta_i$. 
This is a homotopy equivalence. A homotopy inverse 
assigns to $v\in\sV_a$ the unique co-exact triple 
of $1$-forms $\beta\in\pi_a^{-1}(v)$.

Consider the vector bundle 
$$
X:=T^*M\oplus T^*M\oplus T^*M
$$
over $M$ and denote by $X^{(1)}$ the $1$-jet bundle.  
Use the Riemannian metric on $M$ to identify $X^{(1)}$ with 
the set of tuples $(y,\beta_1,\beta_1,\beta_3,L_1,L_2,L_3)$
with $y\in M$, $\beta_i\in T_y^*M$, and $L_i\in\Hom(T_yM,T_y^*M)$. 
Denote by 
$$
\sR_a\subset X^{(1)}
$$
the open subset of all $(y,\beta,L)\in X^{(1)}$
such that the $2$-forms $\tau_i\in\Lambda^2T_y^*M$, 
defined by
\begin{equation}\label{eq:taui}
\tau_i(u,v):=\sigma_i(u,v)+\inner{L_i(u)}{v}-\inner{L_i(v)}{u},\qquad
i=1,2,3,
\end{equation}
are linearly independent.   
Denote by $\sS_a$ the space of sections of $\sR_a$.  
Thus an element of $\sS_a$ is a tuple 
$(\beta,L)=(\beta_1,\beta_2,\beta_3,L_1,L_2,L_3)$
with $\beta_i\in\Om^1(M)$ and $L_i\in\Om^1(M,T^*M)$ 
such that the $2$-forms $\tau_i\in\Om^2(M)$, 
defined by~\eqref{eq:taui} are everywhere linearly independent.  
Then $\sS_a$ is a bundle over $\sF$.  
The projection $\pi_a:\sS_a\to\sF$ is given by $\pi_a(\beta,L)=v$, 
where $\iota(v_i)\dvol_M:=\tau_i$ and $\tau_i$ is as in~\eqref{eq:taui}.  
This map is a homotopy equivalence. 
A homotopy inverse of $\pi_a$ is the inclusion 
$\iota_a:\sF\to\sS_a$ given by $\iota_a(v):=(0,L)$, 
where $L_i(u):=\tfrac12\left(\dvol_M(v_i,u,\cdot)-\sigma_i(u,\cdot)\right)$.
Namely, $\pi_a\circ\iota_a=\id:\sF\to\sF$, both maps are
linear between open subsets of topological vector spaces,
and the kernel of $\pi_a$ is the space of tuples $(\beta,L)$
such that each $L_i$ is symmetric. 

The previous discussion shows that there is a commutative diagram
$$
\xymatrix    
@C=50pt    
@R=25pt    
{    
\sH_a\ar[r]^{\sD_a} \ar[d]_{\pi_a}  & \sS_a\ar[d]^{\pi_a} \\    
\sV_a\ar[r] & \sF 
},
$$
where the vertical maps are homotopy equivalences 
and the differential operator $\sD_a:\sH_a\to\sS_a$ 
is given by $\sD_a\beta:=(\beta,\nabla\beta)$. 
Thus $\sH_a$ is the space of all sections $\beta$ of $X$ 
such that $\sD_a\beta$ satisfies the differential relation $\sR_a$. 
By Lemma~\ref{le:ample}, $\sR_a$ is ample 
in the sense of~\cite[page~167]{EM}. 
Hence $\sR_a$ satisfies the h-principle 
(see~\cite[Theorem~18.4.1]{EM}).  
In particular, every section of $\sR_a$ is homotopic, through
sections of $\sR_a$, to a section of the form $(\beta,\nabla\beta)$.
Equivalently, every frame $v\in\sF$ can be deformed 
within $\sF$ to a divergence-free frame in $\sV_a$.
In fact, by the parametric h-principle, the inclusion 
$\sD_a:\sH_a\to\sS_a$ induces isomorphisms 
on all homotopy groups, and is therefore a homotopy 
equivalence (see~\cite[6.2.A]{EM}).  Hence the
inclusion $\sV_a\hookrightarrow\sF$ is a homotopy equivalence. 

To explain the extension of this result 
to the inclusion of $\sV$ into $\sF$,
it is convenient to spell out the details of the parametric 
h-principle in the present setting. Choose a smooth manifold 
$\Lambda$ and a smooth map $a:\Lambda\to H^2(M;\R)^3$.  
Consider the vector bundle
$$
\tX:=\Lambda\times T^*M\oplus T^*M\oplus T^*M
\to\tM:=\Lambda\times M.  
$$
Define $\tsR\subset\tX^{(1)}$ as the set of tuples
$(\lambda,y,\beta_1,\beta_2,\beta_3,\tL_1,\tL_2,\tL_3)$,
with $\lambda\in\Lambda$, $y\in M$, $\beta_i\in T^*_yM$,
and $\tL_i\in\Hom(T_\lambda\Lambda\times T_yM,T_y^*M)$, 
such that the $2$-forms $\tau_i=\tau_{\lambda,i}\in\Lambda^2T_y^*M$ 
in~\eqref{eq:taui} are linearly independent.
Here $\sigma_i=\sigma_{\lambda,i}$ is the harmonic 
representative of the class $a_i(\lambda)$ and 
$L_i\in\Hom(T_yM,T_y^*M)$ is the restriction of 
$\tL_i$ to $0\times T_yM$. Define the operator 
$\tsD$ from sections of $\tX$ to sections of $\tX^{(1)}$
as the covariant derivative 
$$
\tsD\beta:=(\beta,\nabla\beta).
$$
Let $\tsS$ be the space of sections of $\tsR\subset\tX^{(1)}$
and denote by $\tsH$ its preimage under $\tsD$. 
Thus an element of $\tsH$ is a map
$\Lambda\to\Om^1(M)^3:\lambda\mapsto\beta_\lambda$
such that the $2$-forms 
\begin{equation}\label{eq:tsH}
\tau_{\lambda,i}:=\sigma_{\lambda,i}+d\beta_{\lambda,i},\qquad i=1,2,3,
\end{equation}
are everywhere linearly independent for every $\lambda$. 
An element of $\tsS$ is a smooth 
section that assigns to $\lambda\in\Lambda$ a tuple
\begin{equation}\label{eq:tsS}
(\beta_{\lambda,1},\beta_{\lambda,2},\beta_{\lambda,3},
\tL_{\lambda,1},\tL_{\lambda,2},\tL_{\lambda,3})
\in \Om^1(M)^3\times\Hom(T_\lambda\Lambda,\Om^1(M,T^*M))^3
\end{equation}
such that the $2$-forms $\tau_{\lambda,i}\in\Om^2(M)$, 
defined by~\eqref{eq:taui}, are everywhere linearly independent
for every $\lambda\in\Lambda$.  As before there is a commutative 
diagram 
$$
\xymatrix    
@C=50pt    
@R=25pt    
{    
\tsH\ar[r]^{\tsD} \ar[d]_{\tpi}  & \tsS\ar[d]^{\tpi} \\    
\tsV\ar[r] & \tsF 
}.
$$
Here $\tsV$ is the space of maps 
$\Lambda\to\sV:\lambda\mapsto v_\lambda$
such that $v_\lambda\in\sV_{a(\lambda)}$ 
and $\tsF$ is the space of 
all smooth maps from $\Lambda$ to $\sF$.  
The projection $\tpi:\tsH\to\tsV$,
respectively $\tpi:\tsS\to\tsF$,
assigns to the section $\lambda\mapsto\beta_\lambda$,
respectively~\eqref{eq:tsS}, the section 
$\lambda\mapsto v_\lambda$ with $\iota(v_{\lambda,i})\dvol_M=\tau_{\lambda,i}$, 
where $\tau_{\lambda,i}$ is given by~\eqref{eq:tsH}, 
respectively~\eqref{eq:taui}.  
Both projections are homotopy equivalences.

By Lemma~\ref{le:ample}, the open differential relation $\tsR$ 
is ample.  Hence it follows from the h-principle 
in~\cite[Theorem~18.4.1]{EM} that
every smooth map from $\Lambda$ to $\sF$ 
can be deformed within $\tsF$ to a smooth map 
$\Lambda\to\sV:\lambda\mapsto v_\lambda$ 
that satisfies $v_\lambda\in\sV_{a(\lambda)}$.  
With $\Lambda=S^k$ this implies
that the homomorphism $\pi_k(\sV_a)\to\pi_k(\sF)$ 
is surjective for all $a_1,a_2,a_3\in H^2(M;\R)$.

The relation $\tsR$ also satisfies 
the relative $h$-principle in~\cite[6.2.C]{EM}. 
For the $(k+1)$-ball $\Lambda=B^{k+1}$
with boundary $\p B^{k+1}=S^k$ this means that,
if a map $S^k\to\sV\subset\sF$ extends over $B^{k+1}$ in $\sF$,
and one chooses any smooth extension of the projection 
$S^k\to\sV\to H^2(M;\R)^3$ over $B^{k+1}$,  
then this extension lifts to a smooth map 
$B^{k+1}\to\sV$, equal to the given map over the 
boundary (and homotopic to the given map in $\sF$). 
Hence the homomorphism $\pi_k(\sV)\to\pi_k(\sF)$ 
is injective. This proves Theorem~\ref{thm:gromov}.
\end{proof}

%%%%%%%%%%%%%%%%%%%%%%%%%%%%%%%%%%%%%%%%%%%%%% 
%%%%%%%%%%%%%%%%%%%%%%%%%%%%%%%%%%%%%%%%%%%%%% 
%%%%%%%%%%%%%%%% Appendix B %%%%%%%%%%%%%%%%%% 
%%%%%%%%%%%%%%%%%%%%%%%%%%%%%%%%%%%%%%%%%%%%%% 
%%%%%%%%%%%%%%%%%%%%%%%%%%%%%%%%%%%%%%%%%%%%%% 

\section{Self-Adjoint Fredholm operators} \label{app:RC}  

This appendix is included for the benefit of the reader.
It discusses two well known results about self-adjoint 
Fredholm operators, that are used in Section~\ref{sec:FUETER}.
Lemma~\ref{le:fredholm} characterizes unbounded self-adjoint 
Fredholm operators and Lemma~\ref{le:RC} shows that regular 
crossings are isolated. While Lemma~\ref{le:RC} follows from the 
Kato selection theorem (see~\cite[Lemma~4.7]{RS2}),
the proof given below is simpler and more direct.

Let $H$ be a Hilbert space and $V\subset H$ be a dense linear
subspace that is a Hilbert space in its own right.  Suppose that 
the inclusion $V\hookrightarrow H$ is a compact operator.  
Denote the inner product on $H$ by $\langle\cdot,\cdot\rangle$,
the norm on $H$ by $\Norm{x}_H:=\sqrt{\langle x,x\rangle}$ for $x\in H$,
and the norm on $V$ by $\Norm{x}_V$ for $x\in V$. 
Let $\sS$ be the space of symmetric bounded linear operators $A:V\to H$
and $\sD\subset\sS$ be the subset of self-adjoint operators.
Thus a bounded linear operator $D:V\to H$ is an element of $\sD$
if and only if $\inner{Dx}{\xi}=\inner{x}{D\xi}$ for all $x,\xi\in V$
and, for every $x\in H$, the following holds
\begin{equation}\label{eq:selfadjoint}
\sup_{0\ne \xi\in V}\frac{\Abs{\inner{x}{D\xi}}}{\Norm{\xi}_H}<\infty
\qquad\iff\qquad x\in V.
\end{equation}
Every $D\in\sD$ is a Fredholm operator of index zero and 
regular crossings of differentiable paths $\R\to\sD:s\mapsto D(s)$ 
are isolated.  Proofs of these well known observations
are included here for completeness of the exposition.

\begin{lemma}\label{le:fredholm}
Let $D\in\sS$.  Then the following are equivalent.

\smallskip\noindent{\bf (i)}
$D\in\sD$.

\smallskip\noindent{\bf (ii)}
$(\im\,D)^\perp\subset V$ and
there is a constant $c>0$ such that, for all $x\in V$,
\begin{equation}\label{eq:fredholm}
\Norm{x}_V\le c\left(\Norm{Dx}_H+\Norm{x}_H\right).
\end{equation}

\smallskip\noindent{\bf (iii)}
$D$ is a Fredholm operator of index zero.

\smallskip\noindent
In particular, $\sD$ is an open subset of $\sS$ in the norm topology.
\end{lemma}

\begin{proof}
We prove that~(i) implies~(ii).  Assume $D\in\sD$. 
By~\eqref{eq:selfadjoint} $(\im\,D)^\perp\subset V$. 
We show that the graph of~$D$
is a closed subspace of~${H\times H}$. Let $x_n\in V$ and 
$x,y\in H$ be such that $\lim_{n\to\infty}\Norm{x-x_n}_H=0$
and $\lim_{n\to\infty}\Norm{y-Dx_n}_H=0$. Then
$
\inner{x}{D\xi} = \lim_{n\to\infty}\inner{x_n}{D\xi}
= \lim_{n\to\infty}\inner{Dx_n}{\xi} = \inner{y}{\xi}
$
for $\xi\in V$. Hence $x\in V$ by~\eqref{eq:selfadjoint} 
and, since $D$ is symmetric, it follows that $Dx=y$.  
Thus $D$ has a closed graph.
Now ${V\to\mathrm{graph}(D):x\mapsto(x,Ax)}$
is a bijective bounded linear operator
and so has a bounded inverse.  
This proves~\eqref{eq:fredholm}.

We prove that~(ii) implies~(iii).
Since $V\hookrightarrow H$ is a compact operator, 
it follows from~\eqref{eq:fredholm} that~$D$ 
has a finite-dimensional kernel and a closed image
(see~\cite[Lemma~A.1.1]{MS}). 
Since $(\im\,D)^\perp\subset V$ and~$D$ is symmetric, 
it follows that $(\im\,D)^\perp=\ker\,D$.  Hence 
$\dim\,\coker\,D=\dim\,\ker\,D$. 

We prove that~(iii) implies~(i). 
Let $D\in\sS$ be a Fredholm operator of index zero.
Then $D$ has a finite-dimensional kernel and a closed image.
Since $D$ is symmetric, $\ker\,D\subset(\im\,D)^\perp$. 
Since $D$ has Fredholm index zero, $\ker\,D = (\im\,D)^\perp$
and hence $\im\,D=(\ker\,D)^\perp$.
Now let $x\in H$ and suppose that there is a constant $c$ 
such that $\Abs{\inner{x}{D\xi}}\le c\Norm{\xi}_H$ for every $\xi\in V$.
By the Riesz representation theorem, there exists an element 
$y\in H$ such that $\inner{x}{D\xi}=\inner{y}{\xi}$ for $\xi\in V$.
Choose $y_0\in\ker\,D$ such that $y-y_0\perp\ker\,D$.
Then $y-y_0\in\im\,D$. Choose $x_1\in V$ such that 
$Dx_1=y-y_0$. Then
$
\inner{x-x_1}{D\xi} 
= \inner{y}{\xi}-\inner{Dx_1}{\xi}
= \inner{y_0}{\xi}
$
for every $\xi\in V$. Given $\xi\in V$ choose $\xi_0\in\ker\,D$ 
such that $\xi-\xi_0\perp\,\ker D$. Then
$$
\inner{x-x_1}{D\xi} 
= \inner{x-x_1}{D(\xi-\xi_0)} 
= \inner{y_0}{\xi-\xi_0}=0.
$$
Hence $x-x_1\in(\im\,D)^\perp=\ker\,D\subset V$ and hence $x\in V$.

Since~(i) and~(iii) are equivalent it follows from the perturbation
theory for Fredholm operators (see~\cite[Theorem~A.1.5]{MS}) 
that~$\sD$ is an open subset of $\sS$ with respect to the norm topology.
This proves Lemma~\ref{le:fredholm}.
\end{proof}

Let $I\subset\R$ be an open interval and $I\to\sD:s\mapsto D(s)$
be a continuous path with respect to the norm topology on $\sD$. 
The path is called {\bf weakly differentiable} if the map
$I\to\R:s\mapsto\inner{x}{D(s)\xi}$ is differentiable for 
every $x\in H$ and every $\xi\in V$. A {\bf crossing} is an element
$s\in I$ such that $D(s)$ has a nontrivial kernel.
A crossing $s\in I$ is called {\bf regular} if the quadratic 
form 
$$
\Gamma_s:\ker\,D(s)\to\R,\qquad
\Gamma_s(\xi):=\inner{\xi}{\dot D(s)\xi},
$$
is nondegenerate. 

\begin{lemma}\label{le:RC}
Let $I\to\sD:s\mapsto D(s)$ be a weakly differentiable path
of self-adjoint operators and let $s_0\in I$ be a regular crossing.
Then there is a $\delta>0$ such that 
$D(s):V\to H$ is bijective for every $s\in I$ with
$0<\Abs{s-s_0}<\delta$. 
\end{lemma}

\begin{proof}
Assume without loss of generality that $s_0=0$.
By Lemma~\ref{le:fredholm} there is a constant $c>0$ such that
\begin{equation}\label{eq:cs}
\Norm{x}_V\le c\left(\Norm{D(s)x}_H+\Norm{x}_H\right)
\end{equation}
for every $x\in V$ and every $s$ in some neighborhood of zero.
Shrinking $I$, if necessary, we may assume that~\eqref{eq:cs}
holds for every $s\in I$.  

Assume, by contradiction, that there 
is a sequence $s_n\in I$ such that $s_n\to0$ and $D(s_n)$
is not injective for every $n$.  Then there is a sequence 
$x_n\in V$ such that $D(s_n)x_n=0$ and $\Norm{x_n}_H=1$.
Thus $\Norm{x_n}_V\le c$ by~\eqref{eq:cs}. 
Passing to a subsequence we may assume that $x_n$ 
converges in $H$ to $x_0$.  Then $\Norm{x_0}_H=1$ and
$\inner{x_0}{D(0)\xi} = \lim_{n\to\infty}\inner{x_n}{D(s_n)\xi} = 0$
for $\xi\in V$. Hence $x_0\in\ker\,D(0)$.
Moreover, for every $\xi\in\ker\,D(0)$, the sequence $D(s_n)\xi/s_n$ converges 
weakly to $\dot D(0)\xi$ and is therefore bounded, so
$$
\inner{\dot D(0)\xi}{x_0}
= \lim_{n\to\infty}\Inner{\frac{D(s_n)\xi}{s_n}}{x_0}
= \lim_{n\to\infty}\Inner{\frac{D(s_n)\xi}{s_n}}{x_n}
= 0.
$$
This contradicts the nondegeneracy of~$\Gamma_0$
and proves Lemma~\ref{le:RC}.
\end{proof}

Let $I$ be a compact interval and $I\to\sD:s\mapsto D(s)$
be a weakly differentiable path with only regular crossings
such that $D(s)$ is bijective at the endpoints of $I$. 
The {\bf spectral flow} is the sum of the signatures of 
the crossing forms $\Gamma_s$ over all crossings.
It is invariant under homotopy with fixed endpoints and is additive 
under catenation.  (See~\cite{RS2} for an exposition.)

\medskip\noindent{\bf Acknowledgement.}
These are expanded notes for a lecture given at Imperial College 
on 10 February 2012.  Thanks to Simon Donaldson and Thomas Walpuski 
for many helpful discussions and for their hospitality during my visit.
Thanks to Yasha Eliashberg for his explanation of the h-principle,
and thanks to the referee and editor for helpful comments on the exposition.

%%%%%%%%%%%%%%%%%%%%%%%%%%%%%%%%%%%%%%%%%%%%%%    
%%%%%%%%%%%%%%%%%%%%%%%%%%%%%%%%%%%%%%%%%%%%%%    
%%%%%%%%%%%%%%%% References %%%%%%%%%%%%%%%%%%    
%%%%%%%%%%%%%%%%%%%%%%%%%%%%%%%%%%%%%%%%%%%%%%    
%%%%%%%%%%%%%%%%%%%%%%%%%%%%%%%%%%%%%%%%%%%%%%    

\end{document}